\documentclass[12pt,a4paper]{amsart}
\usepackage{amsmath, amsthm, amsfonts, amssymb, txfonts, graphicx, hyperref, bm,url, verbatim,siunitx}

\theoremstyle{plain}
\newtheorem{thrm}{Theorem}[section]
\newtheorem{lmm}[thrm]{Lemma}
\newtheorem{prpstn}[thrm]{Proposition}

\newtheorem*{cnjctr}{Conjecture}

\newtheorem{hypthss}{Hypothesis}

\numberwithin{sblmm}{thrm} 
\numberwithin{equation}{section}

\renewcommand{\phi}{\varphi}

\begin{document}
\title{Dense clusters of primes in subsets}
\author{James Maynard}
\address{Centre de recherches math\'ematiques,
Universit\'e de Montr\'eal,
Pavillon Andr\'e-Aisenstadt,
2920 Chemin de la tour, Room 5357,
Montr\'eal (Qu\'ebec) H3T 1J4}
\email{maynardj@dms.umontreal.ca}
\begin{abstract}
We prove a generalization of the author's work to show that any subset of the primes which is `well-distributed' in arithmetic progressions contains many primes which are close together. Moreover, our bounds hold with some uniformity in the parameters. As applications, we show there are infinitely many intervals of length $(\log{x})^{\epsilon}$ containing $\gg_\epsilon \log\log{x}$ primes, and show lower bounds of the correct order of magnitude for the number of strings of $m$ congruent primes with $p_{n+m}-p_n\le \epsilon\log{x}$.
\end{abstract}
\maketitle
\section{Introduction}
Let $\mathcal{L}=\{L_1,\dots,L_k\}$ be a set of distinct linear functions $L_i(n)=a_in+b_i$ ($1\le i\le k$) with coefficients in the positive integers. We say such a set is \textit{admissible} if $\prod_{i=1}^kL_i(n)$ has no fixed prime divisor (that is, for every prime $p$ there is an integer $n_p$ such that $\prod_{i=1}^kL_i(n_p)$ is coprime to $p$). Dickson made the following conjecture.
\begin{cnjctr}[Prime $k$-tuples conjecture]
Let $\mathcal{L}=\{L_1,\dots,L_k\}$ be admissible. Then there are infinitely many integers $n$ such that all $L_i(n)$ ($1\le i\le k$) are prime.
\end{cnjctr}
Although such a conjecture appears well beyond the current techniques, recent progress ( \cite{Zhang}, \cite{Maynard}, and unpublished work of Tao) has enabled us to prove weak forms of this conjecture, where instead we show that there are infinitely many integers $n$ such that \textit{several} (rather than \textit{all}) of the $L_i(n)$ are primes.

As noted in \cite{Maynard}, the method of Maynard and Tao can also prove such weak versions of Dickson's conjecture in various more general settings. This has been demonstrated in the recent work \cite{Thorner}, \cite{LemkeOliver}, \cite{Freiberg}, \cite{Pollack}, \cite{Hongze}. In this paper we consider generalized versions of Dickson's conjecture, and prove corresponding weak versions of them.

Based on heuristics from the Hardy-Littlewood circle method, it has been conjectured that the number of $n\le x$ such that all the $L_i(n)$ are prime should have an asymptotic formula $(\mathfrak{S}(\mathcal{L})+o(1))x/(\log{x})^k$, where $\mathfrak{S}(\mathcal{L})$ is a constant depending only on $\mathcal{L}$ (with $\mathfrak{S}(\mathcal{L})>0$ iff $\mathcal{L}$ is admissible). Moreover, these heuristics would suggest that the formulae should hold even if we allow the coefficients $a_i,b_i$ and the number $k$ of functions in $\mathcal{L}$ to vary slightly with $x$.

One can also speculate that Dickson's conjecture might hold for more general sets, where we ask for infinitely many integers $n\in\mathcal{A}$ such that all of $L_i(n)$ are primes in $\mathcal{P}$, for some `nice' sets of integers $\mathcal{A}$ and of primes $\mathcal{P}$, and provided $\mathcal{L}$ satisfies some simple properties in terms of $\mathcal{A}$ and $\mathcal{P}$. For example, Schinzel's Hypothesis H would imply this if either $\mathcal{A}$ or $\mathcal{P}$ are restricted to the values given by an irreducible polynomial, and a uniform version of Dickson's conjecture would give this if $\mathcal{A}$ or $\mathcal{P}$ were restricted to the union of short intervals.

The aim of this paper is to show that the flexibility of the method introduced in \cite{Maynard} allows us to prove weak analogues of these generalizations of Dickson's conjecture. In particular, if $\mathcal{A}$ and $\mathcal{P}\cap L(\mathcal{A})$ are well-distributed in arithmetic progressions, then we can obtain a lower bound close to the expected truth for the number of $n\in\mathcal{A},n\le x$ such that several of the $L_i(n)$ are primes in $\mathcal{P}$, and we can show this estimate holds with some uniformity in the size of $a_i,b_i$ and $k$.
\section{Well-distributed sets}
Given a set of integers $\mathcal{A}$, a set of primes $\mathcal{P}$, and a linear function $L(n)=l_1n+l_2$, we define
\begin{align}
\mathcal{A}(x)&=\{n\in \mathcal{A}: x\le n< 2x\},&\quad \mathcal{A}(x;q,a)&=\{n\in \mathcal{A}(x), n\equiv a\pmod{q}\},\nonumber\\
L(\mathcal{A})&=\{L(n):n\in\mathcal{A}\},\quad &\phi_L(q)&=\phi(|l_1| q)/\phi(|l_1|),\\
\mathcal{P}_{L,\mathcal{A}}(x)&=L(\mathcal{A}(x))\cap\mathcal{P},\quad &\mathcal{P}_{L,\mathcal{A}}(x;q,a)&=L(\mathcal{A}(x;q,a))\cap\mathcal{P}.\nonumber
\end{align}
This paper will focus on sets which satisfy the following hypothesis, which is given in terms of $(\mathcal{A},\mathcal{L},\mathcal{P},B,x,\theta)$ for $\mathcal{L}$ an admissible set of linear functions, $B\in\mathbb{N}$, $x$ a large real number, and  $0<\theta<1$.
\begin{hypthss}\label{hypthss:Weak}$(\mathcal{A},\mathcal{L},\mathcal{P},B,x,\theta)$. Let $k=\#\mathcal{L}$.
\begin{enumerate}
\item $\mathcal{A}$ is well-distributed in arithmetic progressions: We have
\[\sum_{q\le x^{\theta}}\max_{a}\Bigl|\#\mathcal{A}(x;q,a)-\frac{\#\mathcal{A}(x)}{q}\Bigr|\ll \frac{\#\mathcal{A}(x)}{(\log{x})^{100k^2}}.\]
\item Primes in $L(\mathcal{A})\cap \mathcal{P}$ are well-distributed in most arithmetic progressions: For any $L\in\mathcal{L}$ we have
\[\sum_{\substack{q\le x^{\theta}\\ (q,B)=1}}\max_{(L(a),q)=1}\Bigl|\#\mathcal{P}_{L,\mathcal{A}}(x;q,a)-\frac{\#\mathcal{P}_{L,\mathcal{A}}(x)}{\phi_L(q)}\Bigr|\ll \frac{\#\mathcal{P}_{L,\mathcal{A}}(x)}{(\log{x})^{100k^2}}.\]
\item $\mathcal{A}$ is not too concentrated in any arithmetic progression: For any $q<x^\theta$ we have
\[\#\mathcal{A}(x;q,a)\ll \frac{\#\mathcal{A}(x)}{q}.\]
\end{enumerate}
\end{hypthss}
We expect to be able to show this Hypothesis holds (for all large $x$, some fixed $\theta>0$ and some $B<x^{O(1)}$ with few prime factors) for sets $\mathcal{A},\mathcal{P}$ where we can establish `Siegel-Walfisz' type asymptotics for arithmetic progressions to small moduli, and a large sieve estimate to handle larger moduli.

We note that the recent work of Benatar \cite{Benatar} showed the existence of small gaps between primes for sets which satisfy similar properties to those considered here.
\section{Main Results}
\begin{thrm}\label{thrm:MainTheorem}
Let $\alpha>0$ and $0<\theta<1$. Let $\mathcal{A}$ be a set of integers, $\mathcal{P}$ a set of primes, $\mathcal{L}=\{L_1,\dots,L_k\}$ an admissible set of $k$ linear functions, and $B,x$ integers. Let the coefficients $L_i(n)=a_in+b_i\in\mathcal{L}$ satisfy $1\le a_i,b_i\le x^{\alpha}$ for all $1\le i\le k$, and let $k\le (\log{x})^\alpha$ and $1\le B\le x^\alpha$.

There is a constant $C$ depending only on $\alpha$ and $\theta$ such that the following holds. If $k\ge C$ and $(\mathcal{A},\mathcal{L},\mathcal{P},B,x,\theta)$ satisfy Hypothesis \ref{hypthss:Weak}, and if $\delta>(\log{k})^{-1}$ is such that
%
%
\[\frac{1}{k}\frac{\phi(B)}{B}\sum_{L\in\mathcal{L}}\frac{\phi(a_i)}{a_i}\#\mathcal{P}_{L,\mathcal{A}}(x)\ge \delta\frac{\#\mathcal{A}(x)}{\log{x}},\]
then
\[\#\{n\in\mathcal{A}(x):  \#(\{L_1(n),\dots,L_k(n)\}\cap\mathcal{P})\ge C^{-1}\delta\log{k}\}\gg \frac{\#\mathcal{A}(x)}{(\log{x})^{k}\exp(C k)}.\]
Moreover, if $\mathcal{P}=\mathbb{P}$, $k\le (\log{x})^{1/5}$ and all $L\in\mathcal{L}$ have the form $an+b_i$ with $|b_i|\le (\log{x})k^{-2}$ and $a\ll 1$, then the primes counted above can be restricted to be consecutive, at the cost of replacing $\exp(Ck)$ with $\exp(Ck^5)$ in the bound.
\end{thrm}
All implied constants in Theorem \ref{thrm:MainTheorem} are effectively computable if the implied constants in Hypothesis \ref{hypthss:Weak} for $(\mathcal{A},\mathcal{L},\mathcal{P},B,x,\theta)$ are. 

We note that Theorem \ref{thrm:MainTheorem} can show that several of the $L_i(n)$ are primes for sets $\mathcal{A},\mathcal{P}$ where it is \textit{not} the case that there are infinitely many $n\in\mathcal{A}$ such that \textit{all} of the $L_i(n)$ are primes in $\mathcal{P}$. For example, if $\mathcal{P}=\{p_{2n}:n\in\mathbb{N}\}$ is the set of primes of even index and $\mathcal{A}=\mathbb{N}$, then we would expect $\mathcal{P}$ to be equidistributed in the sense of Hypothesis \ref{hypthss:Weak}. However, there are clearly no integers $n$ such that $n,n+2\in\mathcal{P}$, and so the analogue of the twin prime conjecture does not hold in this case. Similarly if $\mathcal{P}$ is restricted to the union of arithmetic progressions in short intervals\footnote{For example, one could take $\mathcal{P}=\cup_{x=2^j}\cup_{i\le x^{1/4}/2}\{x+(2i-1)x^{3/4}<p\le x+2ix^{3/4}]:n\equiv i\pmod{5}\}$. This set is equidistributed in the sense of Hypothesis \ref{hypthss:Weak}, but also has no gaps of size 2.}. Therefore without extra assumptions on our sets $\mathcal{A},\mathcal{P}$ we cannot hope for a much stronger statement than several of the $L_i(n)$ are primes in $\mathcal{P}$.

We also note that Theorem \ref{thrm:MainTheorem} can apply to very sparse sets $\mathcal{A}$, and no density assumptions are required beyond the estimates of Hypothesis \ref{hypthss:Weak}. Of course, for such sets the major obstacle is in establishing Hypothesis \ref{hypthss:Weak}.

We give some applications of this result.
\begin{thrm}\label{thrm:UniformSmallGaps}
For any $x,y\ge 1$ there are $\gg x\exp(-\sqrt{\log{x}})$ integers $x_0\in[x,2x]$ such that
\[\pi(x_0+y)-\pi(x_0)\gg \log{y}.\]
\end{thrm}
Theorem \ref{thrm:UniformSmallGaps} is non-trivial in the region $y=o(\log{x})$ (and $y$ sufficiently large), when typically there are no primes in the interval $[x,x+y]$. For such values of $y$, it shows that there are many intervals of length $y$ containing considerably more than the typical number of primes. By comparison, a uniform version of the prime $k$-tuples conjecture would suggest that for small $y$ there are intervals $[x,x+y]$ containing $\gg y/\log{y}$ primes. For large fixed $y$, we recover the main result of \cite{Maynard}, that $\liminf_{n}(p_{n+m}-p_n)\ll_m 1$ for all $m$.
\begin{thrm}\label{thrm:AveragedShiu}
Fix $\epsilon>0$ and let $x>x_0(\epsilon,q)$. There is a constant $c_\epsilon>0$ (depending only on $\epsilon$) such that uniformly for $m\le c_\epsilon \log\log{x}$, $q\le (\log{x})^{1-\epsilon}$ and $(a,q)=1$ we have
\[\#\{p_n\le x:p_n\equiv \dots \equiv p_{n+m}\equiv a \pmod{q}, p_{n+m}-p_n\le\epsilon\log{x}\}\gg_\epsilon \frac{\pi(x)}{(2q)^{\exp(Cm)}}.\]
Here $C>0$ is a fixed constant.
\end{thrm}
Theorem \ref{thrm:AveragedShiu} extends a result of Shiu \cite{Shiu} which showed the same result but with a lower bound $\gg x^{1-\varepsilon(x)}$ for $\varepsilon(x)\approx C_q m(\log\log{x})^{-1/\phi(q)}$ in the shorter range $m\ll (\log\log{x})^{1/\phi(q)-\epsilon}$ and without the constraint $p_{n+m}-p_n\le \epsilon\log{x}$, and a result of Freiberg \cite{Freiberg2} which showed for fixed $a,q,\epsilon$ infinitely many $n$ such that $p_{n+1}\equiv p_n\equiv a\pmod{q}$ and $p_{n+1}-p_n\le \epsilon \log{p_n}$.

We see that for fixed $m,q$, Theorem \ref{thrm:AveragedShiu} shows a positive proportion of primes $p_n$ are counted (and so our lower bound is of the correct order of magnitude). In particular, for a positive proportion of primes $p_n$ we have\footnote{This disproves the conjecture $\#\{p_n\le x:p_{n}\equiv p_{n+1}\equiv 1\pmod{4}\}=o(\pi(x))$ of Knapowski and Tur\'an \cite{Knapowski}.} $p_n\equiv p_{n+1}\equiv\dots\equiv p_{n+m}\equiv a \pmod{q}$ and $p_{n+m}-p_n\le \epsilon \log{p_n}$. This extends a result of Goldston, Pintz and Y\i ld\i r\i m \cite{GPY:PositiveProportion} which showed a positive proportion of $p_n$ have $p_{n+1}-p_n\le \epsilon\log{p_n}$.

\begin{thrm}\label{thrm:ShortIntervals}
Fix $m\in\mathbb{N}$ and $\epsilon>0$. There exists a $k=\exp(O(m))$, such that for $x>x_0(\epsilon,m)$ and $x^{7/12+\epsilon}\le y\le x$ and for any admissible set $\mathcal{L}=\{L_1,\dots,L_k\}$ where $L_i(n)=a_in+b_i$ with $1\le a_i\ll (\log{x})^{1/\epsilon}$ and $0\le b_i\ll x$ we have
\[\#\{n\in [x,x+y]: \text{at least $m$ of $L_i(n)$ are prime}\}\gg \frac{y}{(\log{x})^{k}}.\]
\end{thrm}
Theorem \ref{thrm:ShortIntervals} relies on a Bombieri-Vinogradov type theorem for primes in intervals of length $x^{7/12+\epsilon}$, the best such result being due to Timofeev \cite{Timofeev}. By adapting Hypothesis \ref{hypthss:Weak} to allow for weighted sums instead of $\#\mathcal{P}_{L,\mathcal{A}}(x)$, we could use presumably the results of \cite{Harman} and \cite{Kumchev} to extend this to the wider range $x^{0.525}\le y\le x$.

Theorem \ref{thrm:ShortIntervals} explicitly demonstrates the claim from \cite{Maynard} that the method also shows the existence of bounded gaps between primes in short intervals, and for linear functions. We note that we would expect the lower bound to be of size $y/(\log{x})^m$, and so our bound is smaller that the expected truth by a factor of a fixed power of $\log{x}$. It appears such a loss is an unavoidable feature of the method when looking at bounded length intervals.

Our final application uses Theorem \ref{thrm:MainTheorem} to apply to a subset $\mathcal{P}$ of the primes. This extends the result of Thorner \cite{Thorner} to sets of linear functions, and with an explicit lower bound.
\begin{thrm}\label{thrm:Chebotarev}
Let $K/\mathbb{Q}$ be a Galois extension of $\mathbb{Q}$ with discriminant $\Delta_K$. There exists a constant $C_K$ depending only on $K$ such that the following holds. Let $\mathcal{C}\subseteq Gal(K/\mathbb{Q})$ be a conjugacy class in the Galois group of $K/\mathbb{Q}$, and let
\[\mathcal{P}=\Bigl\{p\text{ prime}: p\nmid \Delta_K,\, \Bigl[\frac{K/\mathbb{Q}}{p}\Bigr]=\mathcal{C}\Bigr\},\]
where $[\frac{K/\mathbb{Q}}{\cdot}]$ denotes the Artin symbol. Let $m\in\mathbb{N}$ and $k=\exp{(C_Km)}$. For any fixed admissible set $\mathcal{L}=\{L_1,\dots,L_k\}$ of $k$ linear functions $L_i(n)=a_in+b_i$ with $(a_i,\Delta_K)=1$ for each $1\le i\le k$, we have
\[\#\{x\le n\le 2x:\text{at least $m$ of $L_1(n),\dots,L_k(n)$ are in $\mathcal{P}$}\}\gg \frac{x}{(\log{x})^{\exp(C_Km)}},\]
provided $x\ge x_0(K,\mathcal{L})$.
\end{thrm}
Thorner gives several arithmetic consequences of finding such primes of a given splitting type; we refer the reader to the paper \cite{Thorner} for such applications. 

As with Theorem \ref{thrm:ShortIntervals}, we only state the result for fixed $m$, because it relies on other work which establishes the Bombieri-Vinogradov type estimates of Hypothesis \ref{hypthss:Weak}, and these results only save an arbitrary power of $\log{x}$. One would presume these results can be extended to save $\exp(-c\sqrt{\log{x}})$ or similar (having excluded some possible bad moduli), which would allow uniformity for $m\le \epsilon\log\log{x}$, but we do not pursue this here. Similarly, the implied constant in the lower bounds of Theorem \ref{thrm:ShortIntervals} and Theorem \ref{thrm:Chebotarev} is not effective as stated, but presumably a small modification to the underlying results would allow us to obtain an effective bound.
\section{Notation}
We shall view $0<\theta<1$ and $\alpha>0$ as fixed real constants. All asymptotic notation such as $O(\cdot), o(\cdot), \ll, \gg$ should be interpreted as referring to the limit $x\rightarrow\infty$, and any constants (implied by $O(\cdot)$ or denoted by $c,C$ with subscripts) may depend on $\theta,\alpha$ but no other variable, unless otherwise noted. We will adopt the main assumptions of Theorem \ref{thrm:MainTheorem} throughout. In particular we will view $\mathcal{A}$, $\mathcal{P}$ as given sets of integers and primes respectively and $k=\#\mathcal{L}$ will be the size of $\mathcal{L}=\{L_1,\dots,L_k\}$ an admissible set of integer linear functions, and the coefficients $a_i,b_i\in\mathbb{Z}$ of $L_i(n)=a_i n+ b_i$, satisfy $|a_i|,|b_i|\le x^{\alpha}$ and $a_i\ne 0$. $B\le x^\alpha$ will be an integer, and $x,k$ will always to be assumed sufficiently large (in terms of $\theta,\alpha$).

All sums, products and suprema will be assumed to be taken over variables lying in the natural numbers $\mathbb{N}=\{1,2,\dots\}$ unless specified otherwise. The exception to this is when sums or products are over a variable $p$ (or $p'$), which instead will be assumed to lie in the prime numbers $\mathbb{P}=\{2,3,\dots,\}$.

Throughout the paper, $\phi$ will denote the Euler totient function, $\tau_r(n)$ the number of ways of writing $n$ as a product of $r$ natural numbers and $\mu$ the Moebius function. We let $\#\mathcal{A}$ denote the number of elements of a finite set $\mathcal{A}$, and $\mathbf{1}_\mathcal{A}(x)$ the indicator function of $\mathcal{A}$ (so $\mathbf{1}_\mathcal{A}(x)=1$ if $x\in\mathcal{A}$, and 0 otherwise). We let $(a,b)$ be the greatest common divisor of integers $a$ and $b$, and $[a,b]$ the least common multiple of integers $a$ and $b$. (For real numbers $x,y$ we also use $[x,y]$ to denote the closed interval. The usage of $[\cdot,\cdot]$ should be clear from the context.)

To simplify notation we will use vectors in a way which is somewhat non-standard. $\mathbf{d}$ will denote a vector $(d_1,\dots,d_k)\in\mathbb{N}^k$. Given a vector $\mathbf{d}$, when it does not cause confusion, we write $d=\prod_{i=1}^kd_i$. Given $\mathbf{d},\mathbf{e}$, we will let $[\mathbf{d},\mathbf{e}]=\prod_{i=1}^k[d_i,e_i]$ be the product of least common multiples of the components of $\mathbf{d},\mathbf{e}$, and similarly let $(\mathbf{d},\mathbf{e})=\prod_{i=1}^k(d_i,e_i)$ be the product of greatest common divisors of the components, and $\mathbf{d}|\mathbf{e}$ denote the $k$ conditions $d_i|e_i$ for each $1\le i\le k$. An unlabeled sum $\sum_{\mathbf{d}}$ should be interpreted as being over all $\mathbf{d}\in\mathbb{N}^k$.

\section{Outline}
The methods of this paper are based on the `GPY method' for detecting primes. The GPY method works by considering a weighted sum associated to an admissible set $\mathcal{L}=\{L_1,\dots,L_k\}$
\begin{equation}
S=\sum_{x\le n\le 2x}\Bigl(\sum_{i=1}^k\mathbf{1}_{\mathbb{P}}(L_i(n))-m\Bigr)w_n,\label{eq:BasicGPY}
\end{equation}
where $m$ and $k$ are fixed integers, $x$ is a large positive number and $w_n$ are some non-negative weights (typically chosen to be of the form of the weights in Selberg's $\Lambda^2$ sieve).

If $S>0$, then at least one integer $n$ must make a positive contribution to $S$. Since the weights $w_n$ are non-negative, if $n$ makes a positive contribution then the term in parentheses in \eqref{eq:BasicGPY} must be positive at $n$, and so at least $m+1$ of the $L_i(n)$ must be prime. Thus to show at least $m+1$ of the $L_i(n)$ are simultaneously prime infinitely often, it suffices to show that $S>0$ for all large $x$.

The shape of $S$ means that one can consider the terms weighted by $\mathbf{1}_{\mathbb{P}}(L_i(n))$ separately for each $L_i\in\mathcal{L}$, which makes these terms feasible to estimate accurately using current techniques. In particular, the only knowledge about the joint behaviour of the prime values of the $L_i$ is derived from the pigeonhole principle described above.

The method only succeeds if the weights $w_n$ are suitably concentrated on integers $n$ when many of the $L_i(n)$ are prime. To enable an unconditional asymptotic estimate for $S$, the $w_n$ are typically chosen to mimic sieve weights, and in particular Selberg sieve weights (which tend to be the best performing weights when the `dimension' $k$ of the sieve is large). One can then hope to estimate a quantity involving such sieve weights provided one can prove suitable equidistribution results in arithmetic progressions. The strength of concentration of the weights $w_n$ on primes depends directly on the strength of equidistribution results available.

The original work of Goldston Pintz and Y\i ld\i r\i m showed that one could construct weights $w_n$ which would show that $S>0$ for $m=1$ (and for $k$ sufficiently large) if one could prove a suitable extension of the Bombieri-Vinogradov theorem. Zhang \cite{Zhang} succeeded in proving such an extension\footnote{The actual form of Zhang's extension is slightly weaker than that considered in original conditional result of Goldston, Pintz and Y\i ld\i r\i m, although it is sufficient for the argument.}, and as a consequence showed the existence of bounded gaps between primes.

The author's work \cite{Maynard} introduced a modification to the choices of the sieve weights $w_n$ (this modification was also independently discovered by Terence Tao at the same time). This modification enables $w_n$ to be rather more concentrated on $n$ for which many of the $L_i(n)$ are prime. This allows one to show $S>0$ for any $m\in\mathbb{N}$, and moreover the method works even if one has much more limited knowledge about primes in arithmetic progressions.

As remarked in \cite{Maynard}, the fact that the method now works even with only a limited amount of knowledge about primes in arithmetic progressions makes it rather flexible, and in particular applicable to counting primes in subsets, where we have more limited equidistribution results. Moreover, it is possible to exploit the flexibility of the the pigeonhole principle setup in \eqref{eq:BasicGPY} to consider slightly more exotic combinations, which can ensure that the $n$ making a positive contribution to $S$ also satisfy `typical' properties.

Therefore we can consider modified sums of the form
\[S=\sum_{n\in\mathcal{A}(x)}\Bigl(\sum_{i=1}^k\mathbf{1}_\mathcal{P}(L_i(n))-m-k\mathbf{1}_\mathcal{B}(n)\Bigr)w_n\]
for some set of integers $\mathcal{A}$, set of primes $\mathcal{P}$ and set of `atypical' integers $\mathcal{B}$. Provided we have some weak distribution results available (such as those asserted by Hypothesis \ref{hypthss:Weak}) then we can estimate all the terms involved in this sum. Again, by the pigeonhole principle, we see that if $n\in\mathcal{A}(x)$ makes a positive contribution to $S$, then at least $m+1$ of the $L_i(n)$ are primes in $\mathcal{P}$, and that $n\notin\mathcal{B}$. We expect that if $\mathcal{B}$ represents an `atypical' set, and $\mathcal{P}$ is not too sparse (relative to $\mathcal{A}$) then we can choose $w_n$ similarly to before and show that $S>0$ for $k$ sufficiently large. Moreover, by modifying some of the technical aspects of the method in \cite{Maynard}, we can obtain suitable uniform estimates for such sums $S$ even when we allow the coefficients $a_i,b_i$ of $L_i(n)=a_i n+b_i$, the number $k$ of functions and the number $m$ of primes we find to vary with $x$ in certain ranges.

Our work necessarily builds on previous work in \cite{Maynard}, and a certain degree of familiarity with \cite{Maynard} is assumed.

\section{Proof of theorems \ref{thrm:MainTheorem}, \ref{thrm:UniformSmallGaps}, \ref{thrm:AveragedShiu}, \ref{thrm:ShortIntervals} and \ref{thrm:Chebotarev}}
The proof of theorems \ref{thrm:MainTheorem}-\ref{thrm:Chebotarev} relies on the following key proposition.
\begin{prpstn}\label{prpstn:MainProp}
Let $\alpha>0$ and $0<\theta<1$. Let $\mathcal{A}$ be a set of integers, $\mathcal{P}$ a set of primes, $\mathcal{L}=\{L_1,\dots,L_k\}$ an admissible set of $k$ linear functions, and $B,x$ integers. Assume that the coefficients $L_i(n)=a_in+b_i\in\mathcal{L}$ satisfy $|a_i|,|b_i|\le x^{\alpha}$ and $a_i\ne 0$ for all $1\le i\le k$, and that $k\le (\log{x})^\alpha$ and $1\le B\le x^\alpha$. Let $x^{\theta/10}\le R\le x^{\theta/3}$. Let $\rho,\xi$ satisfy $k(\log\log{x})^2/(\log{x})\le \rho,\xi\le \theta/10$, and define
\[\mathcal{S}(\xi;D)=\{n\in\mathbb{N}:p|n\implies (p> x^\xi\text{ or }p|D)\}.\]
There is a constant $C$ depending only on $\alpha$ and $\theta$ such that the following holds. If $k\ge C$ and $(\mathcal{A},\mathcal{L},\mathcal{P},B,x,\theta)$ satisfy Hypothesis \ref{hypthss:Weak}, then there is a choice of nonnegative weights $w_n=w_n(\mathcal{L})$ satisfying

\[w_n\ll (\log{R})^{2k}\prod_{i=1}^k\prod_{p|L_i(n), p\nmid B}4\]
 such that
\begin{enumerate}
\item We have \[\sum_{n\in\mathcal{A}(x)}w_n=\Bigl(1+O\Bigl(\frac{1}{(\log{x})^{1/10}}\Bigr)\Bigr)\frac{B^k}{\phi(B)^k}\mathfrak{S}_B(\mathcal{L})\#\mathcal{A}(x)(\log{R})^k I_k.\]
\item For any $L(n)=a_Ln+b_L\in\mathcal{L}$ with $L(n)>R$ on $[x,2x]$, we have
\begin{align*}
\sum_{n\in\mathcal{A}(x)}\mathbf{1}_\mathcal{P}(L(n))w_n\ge\Bigl(1+O\Bigl(\frac{1}{(\log{x})^{1/10}}\Bigr)\Bigr)\frac{B^{k-1}}{\phi(B)^{k-1}}\mathfrak{S}_B(\mathcal{L})\frac{\phi(|a_L|)}{|a_L|}\#\mathcal{P}_{L,\mathcal{A}}(x)(\log{R})^{k+1} J_{k}\\
+O\Bigl(\frac{B^k}{\phi(B)^k}\mathfrak{S}_B(\mathcal{L})\#\mathcal{A}(x)(\log{R})^{k-1} I_k\Bigr).
\end{align*}
\item For $L=a_0n+b_0\notin\mathcal{L}$ and $D\le x^{\alpha}$, if $\Delta_L\ne 0$ we have
\[\sum_{n\in\mathcal{A}(x)}\mathbf{1}_{\mathcal{S}(\xi;D)}(L(n))w_n\ll \xi^{-1}\frac{\Delta_L}{\phi(\Delta_L)}\frac{D}{\phi(D)}\frac{B^k}{\phi(B)^k}\mathfrak{S}_B(\mathcal{L})\#\mathcal{A}(x)(\log{R})^{k-1} I_{k},\]
where
\[\Delta_L=|a_0|\prod_{j=1}^k|a_0b_j-b_0a_j|.\]
\item For $L\in\mathcal{L}$ we have
\[\sum_{n\in\mathcal{A}(x)}\Bigl(\sum_{\substack{p|L(n)\\p<x^\rho\\ p\nmid B}}1\Bigr)w_n\ll \rho^2 k^4(\log{k})^2\frac{B^k}{\phi(B)^k}\mathfrak{S}_B(\mathcal{L})\#\mathcal{A}(x)(\log{R})^{k} I_{k}.\]
\end{enumerate}
Here $I_k$, $J_k$ are quantities depending only on $k$, and $\mathfrak{S}_B(\mathcal{L})$ is a quantity depending only on $\mathcal{L}$, and these satisfy
\begin{align*}
\mathfrak{S}_B(\mathcal{L})&=\prod_{p\nmid B}\Bigl(1-\frac{\#\{1\le n\le p:p|\prod_{i=1}^k L_i(n)\}}{p}\Bigr)\Bigl(1-\frac{1}{p}\Bigr)^{-k}\gg \frac{1}{\exp(O(k))},\\
I_k&=\int_0^\infty\dotsi\int_0^\infty F^2(t_1,\dots,t_k) dt_1\dots dt_k\gg (2k\log{k})^{-k},\\
J_k&=\int_0^\infty \dots \int_0^\infty \Bigl(\int_0^\infty F(t_1,\dots,t_k)  dt_k\Bigr)^2 dt_1\dots dt_{k-1}\gg \frac{\log{k}}{k}I_k,
\end{align*}
for a smooth function $F=F_k:\mathbb{R}^k\rightarrow \mathbb{R}$ depending only on $k$. Moreover, if all functions $L\in\mathcal{L}$ are of the form $L=an+b_L$, for some fixed $a$ and $b_L\ll \log{x}/(k\log{k})$, then for $\eta\ge (\log{x})^{-9/10}$, we have
\[\sum_{\substack{b\ll \eta\log{x}\\ L(n)=an+b}}\frac{\Delta_L}{\phi(\Delta_L)}\ll \eta(\log{x})(\log{k}).\]
Here the implied constants depend only on $\theta,\alpha,$ and the implied constants from Hypothesis \ref{hypthss:Weak}.
\end{prpstn}
Assuming Proposition \ref{prpstn:MainProp}, we now establish theorems \ref{thrm:MainTheorem}-\ref{thrm:Chebotarev} in turn.
\begin{proof}[Proof of Theorem \ref{thrm:MainTheorem}]
We first note that by passing to a subset of $\mathcal{L}$, it is sufficient to show that in the restricted range $C\le k\le (\log{x})^{1/5}$ we have the weaker bound
\begin{equation}\#\{n\in\mathcal{A}(x):  \#(\{L_1(n),\dots,L_k(n)\}\cap\mathcal{P})\ge C^{-1}\delta\log{k}\}\gg \frac{\#\mathcal{A}(x)}{(\log{x})^{k}\exp(C k^5)}.\label{eq:WeakTheorem}\end{equation}
The main result then follows with a suitably adjusted value of $C$.

For $m\in\mathbb{N}$, we consider the sum
\begin{equation}
S=\sum_{n\in\mathcal{A}(x)}\Bigl(\sum_{i=1}^k\mathbf{1}_{\mathcal{P}}(L_i(n))-m-k\sum_{i=1}^k\sum_{\substack{p|L_i(n)\\p<x^\rho\\ p\nmid B}}1\Bigr)w_n=S_1-S_2-S_3,
\label{eq:SDef}
\end{equation}
where $w_n$ are the weights whose existence is guaranteed by Proposition \ref{prpstn:MainProp}. We note that for any $n\in\mathcal{A}(x)$, the term in parentheses in \eqref{eq:SDef} is positive only if at least $m+1$ of the $L_i(n)$ are primes in $\mathcal{P}$, and none of the $L_i(n)$ have any prime factors $p\nmid B$ less than $x^\rho$. Moreover, we see that if this is the case then since $a_i,b_i<x^\alpha$, each $L_i(n)$ can have at most $O(1/\rho)$ prime factors $p\nmid B$, and so 
\begin{equation}
w_n\ll (\log{x})^{2k}\prod_{i=1}^k\prod_{\substack{p|L_i(n) \\ p\nmid B}}4 \ll (\log{x})^{2k}\exp(O(k/\rho)).
\label{eq:WnBnd}
\end{equation}
Since the term in parentheses in \eqref{eq:SDef} can be at most $k$, we have that
\begin{equation}
\#\{n\in\mathcal{A}(x):  \#(\{L_1(n),\dots,L_k(n)\}\cap\mathcal{P})\ge m\}\gg\frac{S}{k(\log{x})^{2k}\exp(O(k/\rho))}.\label{eq:RealToS}
\end{equation}
Thus it is sufficient to obtain a suitable lower bound for $S$. (Essentially the same idea has been used by Goldston, Pintz and Y\i ld\i r\i m in \cite{GPY:PositiveProportion}.)
Using Proposition \ref{prpstn:MainProp}, we have
\begin{align}
S_1&=\sum_{n\in\mathcal{A}(x)}\sum_{i=1}^k\mathbf{1}_{\mathcal{P}}(L_i(n))w_n\ge(1+o(1))\frac{B^{k-1}}{\phi(B)^{k-1}}\mathfrak{S}_B(\mathcal{L})(\log{R})^{k+1}J_k\sum_{i=1}^k\frac{\phi(a_i)}{a_i}\#\mathcal{P}_{L_i,\mathcal{A}}\\
&\qquad\qquad\qquad\qquad\qquad+o\Bigl(\frac{B^k}{\phi(B)^k}\mathfrak{S}_B(\mathcal{L})\#\mathcal{A}(x)(\log{R})^kI_k\Bigr),\nonumber\\
S_2&=m\sum_{n\in\mathcal{A}(x)}w_n=m(1+o(1))\frac{B^{k}}{\phi(B)^k}\mathfrak{S}_B(\mathcal{L})\#\mathcal{A}(x)(\log{R})^{k}I_k,\\
S_3&=k\sum_{n\in\mathcal{A}(x)}\sum_{i=1}^k\sum_{\substack{p|L_i(n)\\p<x^\rho\\ p\nmid B}}w_n\ll \rho^2 k^6(\log{k})^2\frac{B^k}{\phi(B)^k}\mathfrak{S}_B(\mathcal{L})\#\mathcal{A}(x)(\log{R})^kI_k.
\end{align}
We choose $\rho=c_0 k^{-3}(\log{k})^{-1}$ with $c_0$ a small absolute constant such that $S_3\le (1/3+o(1))S_2$. (This choice satisfies the bounds of Proposition \ref{prpstn:MainProp} since $k\le (\log{x})^{1/5}$ and $k$ is taken to be sufficiently large in terms of $\theta$.) Thus, for $x$ sufficiently large, we have
\begin{align}
S\ge\frac{B^k}{\phi(B)^k}\mathfrak{S}_B(\mathcal{L})(\log{R})^k\Bigl(\frac{J_k}{2}\log{R}\sum_{i=1}^k\frac{\phi(a_i)\phi(B)}{a_iB}\#\mathcal{P}_{L_i,\mathcal{A}}(x)-2mI_k\#\mathcal{A}(x)\Bigr).\label{eq:SLowerBound}
\end{align}
By the assumption of Theorem \ref{thrm:MainTheorem}, we have 
\begin{equation}
\frac{1}{k}\sum_{i=1}^k\frac{\phi(a_i)\phi(B)}{a_iB}\#\mathcal{P}_{L_i,\mathcal{A}}(x)\ge \delta \frac{\#\mathcal{A}(x)}{\log{x}}.\label{eq:PrimeBound}
\end{equation}
From Proposition \ref{prpstn:MainProp}, we have $J_k/I_k\gg (\log{k})/k$. Combining this with \eqref{eq:SLowerBound} and \eqref{eq:PrimeBound}, we have (for $x$ sufficiently large)
\begin{equation}
S\ge (\theta/3)^k\frac{B^k}{\phi(B)^k}\mathfrak{S}_B(\mathcal{L})\#\mathcal{A}(x) (\log{x})^k I_k\Bigl(3c_1\delta\log{k}-2m\Bigr),
\end{equation}
for some constant $c_1$ depending only on $\theta$. In particular, if $m=c_1\delta \log{k}$, then $m\gg 1$ (since $\delta\ge(\log{k})^{-1}$ by assumption), and $S>0$. Using the bounds $I_k\gg (2k\log{k})^{-k}$ and $\mathfrak{S}_B(\mathcal{L})\ge \exp(-Ck)$ from Proposition \ref{prpstn:MainProp}, along with the trivial bound $B/\phi(B)\ge 1$, we obtain
\begin{equation}
S\gg (\theta/3)^k\frac{B^{k}}{\phi(B)^{k}}\mathfrak{S}_B(\mathcal{L})\#\mathcal{A}(x)(\log{x})^kI_k\gg \#\mathcal{A}(x)(\log{x})^k\exp(-C_2 k^2),
\end{equation}
for a suitable constant $C_2$ depending only on $\theta$. Combining this with \eqref{eq:RealToS}, and recalling our choices of $m,\rho$ gives for $x\ge C_3$
\begin{equation}
\#\{n\in\mathcal{A}(x):  \#(\{L_1(n),\dots,L_k(n)\}\cap\mathcal{P})\ge c_1\delta \log{k}\}\ge\frac{\#\mathcal{A}(x)\exp(-C_3 k^5)}{(\log{x})^{k}},
\end{equation}
provided $C_3$ is chosen sufficiently large in terms of $\theta$ and $\alpha$. This gives \eqref{eq:WeakTheorem}, and so the first claim of the theorem.

 For the second claim, we have $L_i=an+b_i$ for all $1\le i\le k$, with $a\ll 1$ and $b_i\le \eta (\log{x})$. (We will eventually take $\eta=c_4(k\log{k})^{-1}$, for some fixed $c_4$ which implies the bound in the statement of Theorem \ref{thrm:MainTheorem}.) In place of $S$ we consider
\begin{align}
S'&=\sum_{n\in\mathcal{A}(x)}\Bigl(\sum_{i=1}^k\mathbf{1}_{\mathcal{P}}(L_i(n))-m-k\sum_{i=1}^k\sum_{\substack{p|L_i(n)\\p<x^\rho, p\nmid B}}1-k\sum_{\substack{b\le \eta\log{x}\\ L=an+b\notin\mathcal{L}}}\mathbf{1}_{\mathcal{S}(\theta/10;1)}(L(n))\Bigr)w_n\label{eq:S'Def}\\
&= S_1-S_2-S_3-S_4.\nonumber
\end{align}
The term in parentheses in \eqref{eq:S'Def} is positive only if at least $m$ of the $L_i(n)$ are primes, none of the $L_i(n)$ have a prime factor $p\nmid B$ smaller than $x^{\rho}$, and all integers not in $\{L_1(n),\dots,L_k(n)\}$ of the form $an+b$ with $b\le \eta\log{x}$ have a prime factor less than $x^{\theta/10}$. In particular, there can be no primes in the interval $[an,an+\eta(\log{x})]$ apart from possibly $\{L_1(n),\dots,L_k(n)\}$, and so the primes counted in this way must be consecutive. 

For $S_4$, we notice that $\Delta_L\ne 0$ for all $L$ we consider since any $L$ has the same lead coefficient as the $L_i$ (and so can't be a multiple of one of them). By Proposition \ref{prpstn:MainProp}, we have
\begin{equation}
S_4\ll k\frac{B^k}{\phi(B)^k}\#\mathcal{A}(x)(\log{R})^{k-1}\mathfrak{S}_B(\mathcal{L})I_k\sum_{\substack{b\le \eta\log{x}\\ L=an+b\notin\mathcal{L}}}\frac{\Delta_L}{\phi(\Delta_L)}\ll \eta k(\log{k})S_2.\label{eq:S4Bound}
\end{equation}
We choose $\eta=c_4/(k\log{k})$ for some sufficiently small constant $c_4$ (this satisfies the requirements of Proposition \ref{prpstn:MainProp}). We then see that the bound \eqref{eq:SLowerBound} holds for $S'$ in place of $S$ provided $x,k$ are sufficiently large. The whole argument then goes through as before.
\end{proof}
\begin{proof}[Proof of Theorem \ref{thrm:UniformSmallGaps}]
We note that the result is trivial if $y\gg (\log{x})^2$, $y=O(1)$ or $x=O(1)$ by the pigeonhole principle, Bertrand's postulate and the prime number theorem. Therefore, by changing the implied constant if necessary, it is sufficient to establish the result for $y\le (\log{x})^{1/5}$ with $y$ sufficiently large.

We take $\theta=1/3$, $\mathcal{P}=\mathbb{P}$, $\mathcal{A}=\mathbb{N}$, $\mathcal{L}=\{L_1,\dots,L_k\}$, with $L_i(n)=n+h_i$, where $h_i$ is the $i^{th}$ prime larger than $k$. By the prime number theorem, $h_i\le 2k\log{k}$ for all $i$ (provided $k$ is sufficiently large). This is an admissible set. 

By the Landau-Page theorem (see, for example, \cite[Chapter 14]{Davenport}) there is at most one modulus $q_0\le \exp(2c_1\sqrt{\log{x}})$ such that there exists a primitive character $\chi$ modulo $q_0$ for which $L(s,\chi)$ has a real zero larger than $1-c_2(\log{x})^{-1/2}$ (for suitable fixed constants $c_1,c_2$). If this exceptional modulus $q_0$ exists, we take $B$ to be the largest prime factor of $q_0$, and otherwise we take $B=1$. For all $q\le \exp(c_1\sqrt{\log{x}})$ with $q\ne q_0$ we then have the effective bound (see, for example, \cite[Chapter 20]{Davenport})
\begin{equation}
\phi(q)^{-1}\sideset{}{^*}\sum_{\chi}|\psi(x,\chi)|\ll x\exp(-3c_1\sqrt{\log{x}}),
\end{equation}
where the summation is over all primitive $\chi\mod{q}$ and $\psi(x,\chi)=\sum_{n\le x}\chi(n)\Lambda(n)$. Following a standard proof of the Bombieri-Vinogradov Theorem (see \cite[Chapter 28]{Davenport}, for example), we have
\begin{equation}
\sum_{\substack{q<x^{1/2-\epsilon}\\ (q,B)=1}}\sup_{(a,q)=1}\Bigl|\pi(x;q,a)-\frac{\pi(x)}{\phi(q)}\Bigr|\ll x\exp(-c_1\sqrt{\log{x}})+\log{x}\sum_{\substack{q<\exp(2c_1\sqrt{\log{x}})\\ (q,B)=1}}\sideset{}{^*}\sum_\chi\frac{|\psi'(x,\chi)|}{\phi(q)}.
\end{equation}
With this choice of parameters, we therefore have error terms for part $(ii)$ of Hypothesis \ref{hypthss:Weak} of size $\#\mathcal{A}(x)\exp(-c_3\sqrt{\log{x}})$, and so Hypothesis \ref{hypthss:Weak} holds for $(\mathcal{A},\mathcal{L},\mathcal{P},B,x,1/3)$ for any $k\le (\log{x})^{1/5}$ provided $k$ is sufficiently large, since parts $(i)$ and $(iii)$ are trivial. Moreover, if $q_0$ exists it must be square-free apart from a possible factor of at most 4, and must satisfy $q_0\gg (\log{x})/(\log\log{x})^2$ (from the class number formula). Therefore if $q_0$ exists, $\log\log{x}\ll B\ll \exp(c_1\sqrt{\log{x}})$. Thus, whether or not $q_0$ exists, we have
\begin{equation}
\frac{B}{\phi(B)}=1+O\Bigl(\frac{1}{\log\log{x}}\Bigr).
\end{equation}

We have
\begin{equation}
\#\mathcal{P}_{L,\mathcal{A}}(x)=\frac{(1+o(1))x}{\log{x}}=\frac{(1+o(1))\#\mathcal{A}(x)}{\log{x}},
\end{equation}
and so we may take $\delta=(1+o(1))$ in Theorem \ref{thrm:MainTheorem}. Theorem \ref{thrm:MainTheorem} then gives
\begin{equation}
\#\{x\le n\le 2x: \pi(n+2k\log{k})-\pi(n)\gg \log{k}\}\gg \frac{x}{(\log{x})^{k}\exp(Ck)}.
\end{equation}
Thus, given any $x,y$ suitably large with $y\le (\log{x})^{1/5}$ we can take $k=\lfloor y/(2\log{y})\rfloor$, and see that the above gives the result. All constants we have used are effectively computable.
\end{proof}
\begin{proof}[Proof of Theorem \ref{thrm:AveragedShiu}]
To get lower bounds of the correct order of magnitude, we average over admissible sets. 
We assume without loss of generality that $a$ is reduced modulo $q$, so $1\le a< q$. We then adopt the same set-up as in the proof of Theorem \ref{thrm:UniformSmallGaps} for our choice of $\mathcal{A},\mathcal{P},\theta,R$. If an exceptional modulus $q_0$ exists (as defined in the proof of Theorem \ref{thrm:UniformSmallGaps}), then we take $B$ to be the largest prime factor of $q_0$ coprime to $q$. Since $q\le (\log{x})^{1-\epsilon}$ and $q_0\gg \log{x}$ (with $q_0$ essentially square-free) we have $\log\log{x}\ll_\epsilon B\ll x$ if $q_0$ exists. Thus $B/\phi(B)=1+o(1)$ regardless of whether $q_0$ exists.

Instead of our individual choice of $\mathcal{L}$, we will average over all admissible choices of $\mathcal{L}$ with $\#\mathcal{L}=k$ and where $\mathcal{L}=\{L_1,\dots,L_k\}$ contains functions of the form  $L_i(n)=qn+a+qb_i$ with $qb_i\le \eta\log{x}$. We write $\mathcal{L}(\mathbf{b})$ for such a set given by $b_1,\dots,b_k$. We consider
\begin{align}
S''&=\sum_{\substack{b_1< \dots< b_k\\ qb_k\le \eta\log{x}\\ \mathcal{L}=\mathcal{L}(\mathbf{b})\text{ admissible}}}\sum_{n\in\mathcal{A}(x)}\Bigl(\sum_{i=1}^k\mathbf{1}_{\mathcal{P}}(L_i(n))-m-k\sum_{i=1}^k\sum_{\substack{p|L_i(n)\\p<x^\rho,p\nmid B}}1-k\sum_{\substack{b\le 2\eta\log{x}\\ L=qn+b\notin\mathcal{L}}}\mathbf{1}_{\mathcal{S}(\rho;B)}(L(n))\Bigr)w_n(\mathcal{L}).
\label{eq:SppDef}
\end{align}
Here $w_n(\mathcal{L})$ are the weights given by Proposition \ref{prpstn:MainProp} for the admissible set $\mathcal{L}=\mathcal{L}(\mathbf{b})$. For a given admissible set $\mathcal{L}$, the sum over $n$ is then essentially the same quantity as $S'$ from \eqref{eq:S'Def}, except in the final term in parentheses we are considering elements with no prime factor less than $x^\rho$ instead of $x^{\theta/10}$.

We see the term in parentheses in \eqref{eq:SppDef} is positive only if at least $m$ of the $L_i(n)$ are primes, all the remaining $L_i(n)$ have no prime factors $p\nmid B$ less than $x^\rho$, and all other $qn+b$ with $b\le 2\eta\log{n}$ have a prime factor $p\nmid B$ less than $x^\rho$. We see from this than no $n$ can make a positive contribution from two different admissible sets (since if $n$ makes a positive contribution for some admissible set, the $L_i(n)$ are uniquely determined as the integers in $[qn,qn+\eta\log{x}]$ with no prime factors $p\nmid B$ less than $x^\rho$). By \eqref{eq:WnBnd}, we see that if $n$ makes a positive contribution then $w_n\ll (\log{x})^{2k}\exp(O(k/\rho))$, with the implied bound uniform in $\mathcal{L}(\mathbf{b})$.

As before, we choose $\rho=c_0k^{-3}(\log{k})^{-1}$, which makes the contribution of the third of the terms in parentheses small compared to the second one. Following the argument of the proof of Theorem \ref{thrm:MainTheorem}, using $\mathcal{S}(\rho;B)$ in place of $\mathcal{S}(\theta/10;1)$ increases the size of the contribution of the final term by a factor $O(\rho^{-1})=O( k^3\log{k})$. Thus to show the final term is suitably small, we take $\eta\le \epsilon$ to be a small multiple of $k^{-4}(\log{k})^{-2}$ instead of $1/(k\log{k})$ (which is acceptable for Proposition \ref{prpstn:MainProp}). With these choices, we find that for a suitable constant $c_1$ we have
\begin{equation}
S''\ge (\theta/3)^k\frac{B^k}{\phi(B)^k}\mathfrak{S}_B(\mathcal{L})\#\mathcal{A}(x)(\log{x})^k I_k\sum_{\substack{b_1<\dots< b_k\\ qb_k\le \eta\log{x}\\ \mathcal{L}(\mathbf{b})\text{ admissible}}}\Bigl(3c_1\log{k}-2m\Bigr).
\end{equation}
Therefore, given $m\in\mathbb{N}$ we choose $k=\lceil\exp(m/c_1)\rceil$. With this choice we see that $S''>0$. Using the bounds $I_k\gg (k\log{k})^{-k}$ and $\mathfrak{S}_B(\mathcal{L})\gg \exp(-Ck)$ from Proposition \ref{prpstn:MainProp} and $B^k/\phi(B)^k\ge 1$ we see that for a suitable constant $C_2$ we have
\begin{equation}
S''\gg x(\log{x})^k\exp(-C_2 k^2)\sum_{\substack{b_1<\dots< b_k\\ qb_k\le \eta\log{x}\\ \mathcal{L}\text{ admissible}}}1.
\label{eq:SppBound}
\end{equation}
Thus we are left to obtain a lower bound for the inner sum of \eqref{eq:SppBound}. We see all the $b_i$ lie between 0 and $\eta(\log{x})/q$. We greedily sieve this interval by removing for each prime $p\le k$ in turn any elements from the residue class modulo $p$ which contains the fewest elements. The resulting set has size at least
\begin{equation}
\frac{\eta\log{x}}{q}\prod_{p\le k}\Bigl(1-\frac{1}{p}\Bigr)\gg \frac{\log{x}}{q k^4(\log{k})^3}.\label{eq:IntervalLength}
\end{equation}
Any choice of $k$ distinct $b_i$ from this set will the cause the resulting $\mathcal{L}(\mathbf{b})$ to be admissible. We now recall from the theorem that we are only considering $q\le (\log{x})^{1-\epsilon}$ and $m\le c_\epsilon\log\log{x}$. For a suitably small choice of $c_\epsilon$, we see that $k=\lceil\exp(m/c_1)\rceil\le (\log{x})^{\epsilon/10}$. Therefore from \eqref{eq:IntervalLength} we see the length of the interval is at least $k^2$ if $x$ is sufficiently large in terms of $\epsilon$. In this case, we obtain the bound
\begin{equation}
\sum_{\substack{b_1< \dots< b_k\\ qb_k\le \eta\log{x}\\ \mathcal{L}\text{ admissible}}}1\ge k^{-k}\Bigl(\frac{c_3\log{x}}{qk^4\log^3{k}}-k\Bigr)^k\gg \Bigl(\frac{\log{x}}{q}\Bigr)^k\exp(-C_4 k^2),\label{eq:AdmissibleBound}
\end{equation} 
for some constants $c_3,C_4>0$. Thus, substituting \eqref{eq:AdmissibleBound} into \eqref{eq:SppBound} we obtain
\begin{equation}
S''\gg x(\log{x})^{2k}\exp(-C_5 k^2)q^{-k}.
\end{equation}
We recall that every pair $(n,\mathcal{L})$ for which $n$ makes a positive contribution to $S''$ when considering $\mathcal{L}$ is counted with weight at most $kw_n(\mathcal{L})\ll k(\log{x})^{2k}\exp(O(k/\rho))$ (uniformly over all choices of $\mathcal{L}$). Putting this all together, we obtain the number $N$ of integers $n$ with $x\le n\le 2x$ such that there are $\gg \log{k}$ consecutive primes all congruent to $a\pmod{q}$ in the interval $[qn,qn+\eta\log{x}]$ satisfies
\begin{equation}
N\gg \frac{x}{q^k\exp(C_6 k^5)}.
\end{equation}
We see that the initial prime in each such interval is counted by at most $\log{x}$ values of $n$. Therefore, changing the count to be over the initial prime, recalling $k=\lceil\exp(m/c_1)\rceil$, recalling that $\eta\le \epsilon$, and replacing $x$ with $x/3q$ gives
\begin{equation}
\#\{p_n\le x:p_n\equiv \dots \equiv p_{n+m}\equiv a \pmod{q}, p_{n+m}-p_n\le\epsilon\log{x}\}\gg_\epsilon \frac{\pi(x)}{(2q)^{\exp(Cm)}},
\end{equation}
for a suitable constant $C>0$, as required.
\end{proof}
\begin{proof}[Proof of Theorem \ref{thrm:ShortIntervals}]
We take $\mathcal{P}=\mathbb{P}$, $\mathcal{A}=[x,x+y]$, $B=1$, $\theta=1/30-\epsilon$. Given $m$, we choose $k=\exp(Cm)$ for some suitable constant $C>0$.

Timofeev \cite{Timofeev} (improving earlier work of Huxley and Iwaniec \cite{HuxleyIwaniec} and Perelli, Pintz and Salerno \cite{PerelliPintzSalerno}) has shown that, for $\theta=1/30-\epsilon/2$, for any $x^{7/12+\epsilon/2}\le y\le x$ and any fixed $C'>0$ we have
\begin{equation}
\sum_{q<x^{\theta}}\sup_{(a,q)=1}\Bigl|\pi(x+y;q,a)-\pi(x;q,a)-\frac{\pi(x+y)-\pi(x)}{\phi(q)}\Bigr|\ll _{C',\epsilon} \frac{y}{(\log{x})^{C'}}.
\label{eq:BVShortIntervals}
\end{equation}
By taking $C'$ sufficiently large in terms of $k$, we see that \eqref{eq:BVShortIntervals} implies Hypothesis \ref{hypthss:Weak} holds for our choice of $\theta=1/30-\epsilon$ provided $x$ is sufficiently large in terms of $m$ and $\epsilon$. Theorem \ref{thrm:MainTheorem} then automatically gives Theorem \ref{thrm:ShortIntervals}.
\end{proof}
\begin{proof}[Proof of Theorem \ref{thrm:Chebotarev}]
We take $\mathcal{A}=\mathbb{N}$, $B=\Delta_K$ and $\mathcal{P},\mathcal{L}$ the sets given by the statement of the theorem. To avoid confusion, we note that $\Delta_K$ here is the discriminant of $K/\mathbb{Q}$, and unrelated to $\Delta_L$ from Proposition \ref{prpstn:MainProp}. Murty and Murty \cite{Murty} have then established the key estimate $(2)$ of Hypothesis \ref{hypthss:Weak} with any $\theta<\min(1/2,2/\#G)$, where $G=Gal(K/\mathbb{Q})$ (the other estimates being trivial). Finally, we have
\begin{equation}
\frac{1}{k}\frac{B}{\phi(B)}\sum_{i=1}^k\frac{\phi(a_i)}{a_i}\#\mathcal{P}_{L_i,\mathcal{A}}(x)\ge (1+o(1))\frac{\Delta_K\#\mathcal{C}}{\phi(\Delta_K)\#G}\frac{x}{\log{x}},
\end{equation}
and so for $x$ sufficiently large, we may take $\delta$ to be a constant depending only on $K$. The result now follows directly from Theorem \ref{thrm:MainTheorem}.
\end{proof}
\section{Initial Considerations}
\label{sctn:InitialConsiderations}
We recall that we are given a set $\mathcal{A}$ of integers, a set $\mathcal{P}$ of primes, an admissible set $\mathcal{L}=\{L_1,\dots,L_k\}$ of integer linear functions, an integer $B$ and quantities $R,\,x$. We assume that the coefficients of $L_i(n)=a_in+b_i\in\mathcal{L}$ satisfy $|a_i|,|b_i|\le x^\alpha$, $a_i\ne 0$, and $k=\#\mathcal{L}$ is sufficiently large in terms of the fixed quantites $\theta,\alpha$ and satisfies $k\le (\log{x})^{1/5}$. $B,R$ satisfy $1\le B\le x^\alpha$, and $x^{\theta/10}\le R\le x^{\theta/3}$. Finally, we assume from now on that the set $\mathcal{A}$ satisfies
\begin{equation}
\sum_{q\le x^{\theta}}\max_{a}\Bigl|\#\mathcal{A}(x;q,a)-\frac{\#\mathcal{A}(x)}{q}\Bigr|\ll \frac{\#\mathcal{A}(x)}{(\log{x})^{100k^2}},\label{eq:Bdd1}
\end{equation}
and
\begin{equation}
\#\mathcal{A}(x;q,a)\ll \frac{\#\mathcal{A}(x)}{q}\label{eq:Bdd2}
\end{equation}
for any $q<x^\theta$. Together these assumptions are a slight generalization of the assumptions of Proposition \ref{prpstn:MainProp}.

We define the multiplicative functions $\omega=\omega_\mathcal{L}$ and $\phi_\omega=\phi_{\omega,\mathcal{L}}$ and the singular series $\mathfrak{S}_D(\mathcal{L})$ for an integer $D$ by
\begin{align}
\omega(p)&=\begin{cases}
\#\{1\le n\le p: \prod_{i=1}^k L_i(n)\equiv 0\pmod{p}\},\qquad &p\nmid B,\\
0,&p|B,
\end{cases}\\
\phi_\omega(d)&=\prod_{p|d}(p-\omega(p)),\\
\mathfrak{S}_D(\mathcal{L})&=\prod_{p\nmid D}\Bigl(1-\frac{\omega(p)}{p}\Bigr)\Bigl(1-\frac{1}{p}\Bigr)^{-k}.\label{eq:SingularSeriesDef}
\end{align}
Since $\mathcal{L}$ is admissible, we have $\omega(p)<p$ for all $p$ and so $\phi_\omega(n)>0$ and $\mathfrak{S}_D(\mathcal{L})> 0$ for any integer $D$. Since $\omega(p)=k$ for all $p\nmid \prod_{i=1}^ka_i\prod_{i\ne j}(a_ib_j-b_ia_j)$ we see the product $\mathfrak{S}_D(\mathcal{L})$ converges. 

The main innovation in \cite{Maynard} was a different choice of the sieve weights used in the GPY method to detect small gaps between primes. In order to adapt the argument of \cite{Maynard} to the more general situation considered here, we need to modify the choice of these weights further to produce a choice more amenable to obtaining uniform estimates. In particular, in \cite{Maynard} the `W-trick' was used to eliminate the need for consideration of the singular series which would naturally arise. In our situation, however, in order to obtain suitable uniform estimates without stronger assumptions on the error terms in Hypothesis \ref{hypthss:Weak}, we need to take these singular series into account.

We will consider sieve weights $w_n=w_n(\mathcal{L})$, which are defined to be 0 if $\prod_{i=1}^kL_i(n)$ is a multiple of any prime $p\le 2k^2$ with $p\nmid B$. We let $W=\prod_{p\le 2k^2,p\nmid B}p$. If $(L_i(n),W)=1$ for all $1\le i\le k$ we have
\begin{equation}
w_n=\Bigl(\sum_{d_i|L_i(n)\forall i}\lambda_{\mathbf{d}}\Bigr)^2,
\label{eq:WnDef}
\end{equation}
for some real variables $\lambda_{\mathbf{d}}$ depending on $\mathbf{d}=(d_1,\dots,d_k)$. We first restrict our $\lambda_{\mathbf{d}}$ to be supported on $\mathbf{d}$ with $d=\prod_{i=1}^k d_i$ square-free and coprime to $WB$.

Given a prime $p\nmid WB$, let $1\le r_{p,1}<\dots< r_{p,\omega(p)}\le p$ be the $\omega(p)$ residue classes for which $\prod_{i=1}^k L_i(n)$ vanishes modulo $p$. For each such prime $p$, we fix a choice of indices $j_{p,1},\dots,j_{p,\omega(p)}\in\{1,\dots,k\}$ such that $j_{p,i}$ is the smallest index such that
\begin{equation}
L_{j_{p,i}}(r_{p,i})\equiv0\pmod{p}
\end{equation}
for each $i\in\{1,\dots,\omega(p)\}$. (We could choose any index satisfying the above condition; we choose the smallest index purely for concreteness.) All the functions $L_i$ are linear and, since $\mathcal{L}$ is admissible, none of the $L_i$ are a multiple of $p$. This means that for any $L\in\mathcal{L}$ there is at most one residue class for which $L$ vanishes modulo $p$. Thus the indices $j_{p,1},\dots, j_{p,\omega(p)}$ we have chosen must be distinct. We now restrict the support of $\lambda_{\mathbf{d}}$ to $(d_j,p)=1$ for all $j\notin\{j_{p,1},\dots,j_{p,\omega(p)}\}$.

We see these restrictions are equivalent to the restriction that the support of $\lambda_{\mathbf{d}}$ must lie the set
\begin{equation}
\mathcal{D}_k=\mathcal{D}_k(\mathcal{L})=\{\mathbf{d}\in\mathbb{N}^k:\mu^2(d)=1, (d_j,W_j)=1\forall j\},
\label{eq:DkDef}
\end{equation}
where $W_j$ are square-free integers each a multiple of $WB$, and any prime $p\nmid WB$ divides exactly $k-\omega(p)$ of the $W_j$ (such $p|W_j$ if $j\notin\{j_{p,1},\dots,j_{p,\omega(p)}\}$). We recall that in our notation $\mu^2(d)=\mu^2(\prod_{i=1}^kd_i)$.

The key point of these restrictions is so that different components of different $\mathbf{d}$ occurring in our sieve weights will be relatively prime. Indeed, let $\mathbf{d}$ and $\mathbf{e}$ both occur in the sum \eqref{eq:WnDef}. If $p|d_i$ then $p|L_i(n)$, and so $i$ must be the chosen index for the residue class $n$ $\pmod{p}$. But if we also have $p|e_j$ then similarly $j$ must be the chosen index for this residue class, and so we must have $i=j$. Hence $(d_i,e_j)=1$ for all $i\ne j$.

Similar to \cite{Maynard}, we define $\lambda_{\mathbf{d}}$ in terms of variables $y_{\mathbf{r}}$ supported on $\mathbf{r}\in\mathcal{D}_k$ by
\begin{equation}
\lambda_{\mathbf{d}}=\mu(d)d\sum_{\mathbf{d}|\mathbf{r}}\frac{y_{\mathbf{r}}}{\phi_\omega(r)},\qquad y_{\mathbf{r}}=\frac{\mathbf{1}_{\mathcal{D}_k}(\mathbf{r})W^kB^k}{\phi(W B)^k}\mathfrak{S}_{WB}(\mathcal{L})F\Bigl(\frac{\log{r_1}}{\log{R}},\dots,\frac{\log{r_k}}{\log{R}}\Bigr),\label{eq:YDef}
\end{equation}
(again, we recall $d=\prod_{i=1}^kd_i$) where $x^{\theta/10}\le R\le x^{\theta/3}$ and $F:\mathbb{R}^k\rightarrow\mathbb{R}$ is a smooth function given by
\begin{equation}
F(t_1,\dots,t_k)=\psi\Bigl(\sum_{i=1}^kt_i\Bigr)\prod_{i=1}^k\frac{\psi(t_i/U_k)}{1+T_kt_i},\qquad T_k=k\log{k},\qquad U_k=k^{-1/2}.\label{eq:FDef}
\end{equation}
Here $\psi:[0,\infty)\rightarrow[0,1]$ is a fixed smooth non-increasing function supported on $[0,1]$ which is $1$ on $[0,9/10]$. In particular, we note that this choice of $F$ is non-negative, and that the support of $\psi$ implies that
\begin{equation}
\lambda_{\mathbf{d}}=0\quad\text{if $\textstyle d=\prod_{i=1}^kd_i>R$.}
\end{equation}
We will find it useful to also consider the closely related functions $F_1$ and $F_2$ which will appear in our error estimates, defined by
\begin{equation}
 F_1 (t_1,\dots,t_k)=\prod_{i=1}^k \frac{\psi(t_i/U_k)}{1+T_kt_i},\qquad 
 F_2 (t_1,\dots,t_k)=\sum_{1\le j\le k}\Bigl(\frac{\psi(t_j/2)}{1+T_kt_j}\prod_{\substack{1\le i\le k\\ i\ne j}}\frac{\psi(t_i/U_k)}{1+T_kt_i}\Bigr).
\label{eq:F1F2Def}
\end{equation}
Finally, by Moebius inversion, we see that \eqref{eq:YDef} implies that for $\mathbf{r}\in\mathcal{D}_k$
\begin{equation}
y_{\mathbf{r}}=\mu(r)\phi_\omega(r)\sum_{\mathbf{r}|\mathbf{f}}\frac{y_{\mathbf{f}}}{\phi_\omega(f)}\sum_{\substack{\mathbf{d}\\\mathbf{r}|\mathbf{d},\mathbf{d}|\mathbf{f}}}\mu(d)=\mu(r)\phi_\omega(r)\sum_{\mathbf{r}|\mathbf{d}}\frac{\lambda_{\mathbf{d}}}{d}.
\end{equation}
\section{Preparatory Lemmas}
\begin{lmm}\label{lmm:SingularSeries}
(i) There is a constant $C$, such that for any admissible set $\mathcal{L}$ of size $k$ we have
\[\mathfrak{S}_B(\mathcal{L})\ge \exp(-Ck).\]
(ii) Let all functions $L_i\in\mathcal{L}$ be of the form $L_i=an+b_i$, for some integers $|a|\ll 1$ and $|b_i| \ll \log{x}$. Let $\Delta_L=|a|^{k+1}\prod_{i=1}^k |b_i-b|$ and $\eta\ge (\log{x})^{-9/10}$. Then we have
\[\sum_{\substack{|b|\le \eta\log{x}\\ L(n)=an+b\notin\mathcal{L}}}\frac{\Delta_L}{\phi(\Delta_L)}\ll \eta(\log{x})(\log{k}).\]
\end{lmm}
\begin{proof}
Since $\omega(p)\le \min(k,p-1)$ for any admissible $\mathcal{L}$ of size $k$, we have
\begin{align}
\mathfrak{S}_B(\mathcal{L})=\prod_{p\nmid B}\Bigl(1-\frac{\omega(p)}{p}\Bigr)\Bigl(1-\frac{1}{p}\Bigr)^{-k}\ge \prod_{p\le k, p\nmid B}\frac{1}{p}\prod_{p> k,p\nmid B}\Bigl(1-\frac{k}{p}\Bigr)\Bigl(1-\frac{1}{p}\Bigr)^{-k}.
\end{align}
Since all terms in the products on the right hand side are less than 1, we can drop the restriction $p\nmid B$ for a lower bound. This gives
\begin{equation}
\mathfrak{S}_B(\mathcal{L})\ge \prod_{p\le k}\frac{1}{p}\prod_{p> k}\Bigl(1+O(k^2/p^2)\Bigr)\ge \exp(-Ck).
\end{equation}
We now consider the second statement. We have $L_i(n)=an+b_i$ with $|b_i|\ll \log{x}$, and consider $L=an+b\notin\mathcal{L}$ with $|b|\le \eta\log{x}$. 
If $k\gg \log\log{x}$ then we use 
the bound $\Delta_L/\phi(\Delta_L)\ll\log\log{\Delta_L}\ll \log{k}$ 
to give
\begin{equation}
\sum_{\substack{|b|\le \eta\log{x}\\ L=an+b\notin\mathcal{L}}}\frac{\Delta_L}{\phi(\Delta_L)}\ll \eta(\log{k})(\log{x}).\label{eq:KLarge}
\end{equation}
We now establish \eqref{eq:KLarge} in the case $k\ll\log\log{x}$. Using the identity $e/\phi(e)=\sum_{d|e}\mu^2(d)/\phi(d)$, and splitting the terms depending on the size of divisors, we have
\begin{align}
\sum_{\substack{|b|\le \eta\log{x}\\ L=an+b\notin\mathcal{L}}}\frac{\Delta_L}{\phi(\Delta_L)}&=\frac{a}{\phi(a)} 
\sum_{\substack{|b|\le \eta\log{x}\\ L=an+b\notin\mathcal{L}}}\sum_{\substack{d|\Delta_L\\ (d,a)=1}}\frac{\mu^2(d)}{\phi(d)}\nonumber\\
&\ll 
\sum_{\substack{|b|\le \eta\log{x}\\ L=an+b\notin\mathcal{L}}}\Bigl(\sum_{\substack{1\le d\le \eta\log{x}\\ d|\Delta_L,(d,a)=1}}\frac{\mu^2(d)}{\phi(d)}+\sum_{\substack{d>\eta\log{x}\\ d|\Delta_L}}\frac{\mu^2(d)\sum_{p|d}\log{p}}{\phi(d)\log(\eta\log{x})}\Bigr)\nonumber\\
&\ll  
\sum_{\substack{1\le d\le \eta\log{x}\\(d,a)=1}}\frac{\mu^2(d)}{\phi(d)}\sum_{\substack{|b|\le \eta\log{x}\\ L=an+b\notin\mathcal{L}\\ d|\Delta_L}}1+
\sum_{\substack{|b|\le \eta\log{x}\\ L=an+b\notin\mathcal{L}}}\sum_{p|\Delta_L}\frac{\log{p}}{p\log(\eta\log{x})}\frac{\Delta_L}{\phi(\Delta_L)}.\label{eq:SingularSum}
\end{align}
We first consider the second term on the right hand side of \eqref{eq:SingularSum}. We have $\sum_{p|\Delta_L}p^{-1}\log{p}\ll\log\log{\Delta_L}$ and $\Delta_L/\phi(\Delta_L)\ll \log\log{\Delta_L}$. But we are only we only considering $k\ll\log\log{x}$ and $\eta\ge(\log{x})^{-9/10}$, and so $(\log\log{\Delta_L})^2\ll (\log\log\log{x)^2}=o(\log(\eta\log{x}))$. Therefore we see that the total contribution from the second term in \eqref{eq:SingularSum} is $o(\eta
\log{x})$.

We now consider the first sum in \eqref{eq:SingularSum}. For every prime $p|d$, there are at most $k$ choices for the residue class $b\pmod{p}$ such that $p|\Delta_L$, and trivially there are also at most $p$ choices. Thus the inner sum can be written as $\prod_{p|d}\min(p,k)$ sums over $b$ in a fixed arithmetic progression modulo $d$. For each such sum there are $\ll \eta(\log{x})/d$ possible values of $b$. Thus we have
\begin{align}
\sum_{\substack{1\le d\le \eta\log{x}\\ (d,a)=1}}\frac{\mu^2(d)}{\phi(d)}\sum_{\substack{|b|\le \eta\log{x}\\ L=an+b\notin\mathcal{L}\\ d|\Delta_L}}1&\ll \sum_{d\le \eta\log{x}}\frac{\mu^2(d)\prod_{p|d}\min(p,k)}{\phi(d)}\Bigl(\frac{\eta\log{x}}{d}\Bigr)\nonumber\\
&\ll \eta\log{x}\prod_{p\le k}\Bigl(1+\frac{1}{p-1}\Bigr)\prod_{p>k}\Bigl(1+\frac{k}{p(p-1)}\Bigr)\nonumber\\
&\ll \eta(\log{x})(\log{k}).\label{eq:KSmall}
\end{align}
This gives the result.
\end{proof}
\begin{lmm}\label{lmm:YDifference}
Let
\[Y_\mathbf{r}=\frac{W^kB^k\mathfrak{S}_{WB}(\mathcal{L})}{\phi(W B)^{k}} F_2 \Bigl(\frac{\log{r_1}}{\log{R}},\dots,\frac{\log{r_k}}{\log{R}}\Bigr),\]
where $F_2$ is given by \eqref{eq:F1F2Def}. Then

(i) Let $\mathbf{r},\mathbf{s}\in\mathcal{D}_k$ with $s_i=r_i$ for all $i\ne j$, and $s_j=Ar_j$ for some $A\in\mathbb{N}$. Then
\[y_{\mathbf{s}}=y_{\mathbf{r}}+O\Bigl(T_k Y_{\mathbf{r}}\frac{\log{A}}{\log{R}}\Bigr).\]
 (ii) Let $\mathbf{r},\mathbf{s}\in\mathcal{D}_k$ with $r=s$ and let $A$ be the product of primes dividing $r$ but not $(\mathbf{r},\mathbf{s})$. Then
\[y_{\mathbf{s}}=y_{\mathbf{r}}+O\Bigl(T_k (Y_{\mathbf{r}}+Y_{\mathbf{s}})\frac{\log{A}}{\log{R}}\Bigr).\]
\end{lmm}
\begin{proof}
We recall the definitions of $\psi$, $F_2$, $U_k=k^{-1/2}$ and $T_k=k\log{k}$ from Section \ref{sctn:InitialConsiderations}. Given $u,v\ge 0$ with $|u-v|\le \epsilon$, we have
\begin{equation}
\frac{1}{1+T_ku}= \frac{1+O(T_k\epsilon)}{1+T_kv},\qquad \psi(u)=\psi(v)+O(\epsilon).\label{eq:BasicBounds}
\end{equation}
We let $u_i=\log{r_i}/\log{R}$, $v_i=\log{s_i}/\log{R}$ and $\epsilon_i=v_i-u_i$. For part $(i)$ we have $\epsilon_i=0$ for $i\ne j$ and $\epsilon_j=\log{A}/\log{R}$. We may assume $\epsilon_j\le 1, u_j\le U_k$ since otherwise the result is trivial. By \eqref{eq:BasicBounds} we have
\begin{align}
\psi\Bigl(\sum_{i=1}^k v_i\Bigr)\frac{\psi(v_j/U_k)}{1+T_kv_j}&=\Bigl(\psi\Bigr(\sum_{i=1}^k u_i\Bigr)+O\Bigl(\frac{\log{A}}{\log{R}}\Bigr)\Bigr)\Bigl(\psi\Bigl(\frac{u_j}{U_k}\Bigr)+O\Bigl(\frac{\log{A}}{U_k\log{R}}\Bigr)\Bigr)\frac{1+O\Bigl(T_k\frac{\log{A}}{\log{R}}\Bigr)}{1+T_ku_j}.
\end{align}
Since $1+U_k^{-1}\ll T_k$, $0\le \psi\le 1$ and $\psi(v_j/2)=1$ (since $v_j=u_j+\epsilon_j\le 1+U_k<9/5$), expanding the terms and multiplying by $\prod_{i\ne j}\psi(u_i/U_k)/(1+T_ku_i)$ gives the result for $(i)$. 

We now consider part $(ii)$.
We let $\mathbf{t}$ be the vector with $t_i=[r_i,s_i]$. By applying part $(i)$ to each component in turn, and using the fact that $Y_{\mathbf{r}}$ is decreasing, we find that
\begin{equation}
y_{\mathbf{s}}=y_{\mathbf{t}}+O\Bigl(T_k Y_{\mathbf{s}} \sum_{i=1}^k\frac{\log{[r_i,s_i]/s_i}}{\log{R}}\Bigr)=y_{\mathbf{t}}+O\Bigl(T_k Y_{\mathbf{s}} \frac{\log{A}}{\log{R}}\Bigr).
\end{equation}
We obtain the same expression for $\mathbf{r}$ in place of $\mathbf{s}$, and hence the result follows.
%
\end{proof}

We use the following lemma to estimate the various smoothed sums of multiplicative functions which we will encounter.
\begin{lmm}\label{lmm:PartialSummation}
Let $A_1,A_2,L>0$. Let $\gamma$ be a multiplicative function satisfying
\[0\le \frac{\gamma(p)}{p}\le 1-A_1,\quad\text{and}\quad -L\le \sum_{w\le p\le z}\frac{\gamma(p)\log{p}}{p}- \log{z/w}\le A_2\]
for any $2\le w\le z$. Let $g$ be the totally multiplicative function defined on primes by $g(p)=\gamma(p)/(p-\gamma(p))$. Finally, let $G:[0,1]\rightarrow \mathbb{R}$ be smooth, and let $G_{max}=\sup_{t\in [0,1]}(|G(t)|+|G'(t)|)$. Then
\[\sum_{d<z}\mu(d)^2g(d)G\Bigl(\frac{\log{d}}{\log{z}}\Bigr)=c_\gamma \log{z}\int_0^1G(x)dx+O_{A_1,A_2}(c_\gamma LG_{max}),\]
where
\[c_\gamma=\prod_{p}\Bigl(1-\frac{\gamma(p)}{p}\Bigr)^{-1}\Bigl(1-\frac{1}{p}\Bigr).\]
\end{lmm}
\begin{proof}
This is \cite[Lemma 4]{GGPY}, with $\kappa=1$ and slight changes to the notation.
\end{proof}
\begin{lmm}\label{lmm:MultipleSummation}
Let $r\le k\ll (\log{R})^{1/5}$. Let $W_1,\dots, W_r\le R^{O(k)}$ all be a multiple of $\prod_{p\le 2k^2}p$. Let $g$ be a multiplicative function with $g(p)=p+O(k)$. Let $G:\mathbb{R}\rightarrow\mathbb{R}$ be a smooth function supported on the interval $[0,1]$ such that 
\[\sup_{t\in [0,1]}(|G(t)|+|G'(t)|)\le \Omega_G\int_0^\infty G(t)dt,\]
for some quantity $\Omega_G$ which satisfies $r\Omega_G=o((\log{R})/(\log\log{R}))$.
Let $\Phi:\mathbb{R}\rightarrow\mathbb{R}$ be smooth with $\Phi(t),\Phi'(t)\ll 1$ for all $t$. 

Then for $k$ sufficiently large we have
\begin{align*}
\sum_{\substack{\mathbf{e}\in\mathbb{N}^r\\ (e_i,W_i)=1\forall i }}\frac{\mu^2(e)}{g(e)}&\Phi\Bigl(\sum_{i=1}^k\frac{\log{e_i}}{\log{R}}\Bigr)\prod_{i=1}^k G\Bigl(\frac{\log{e_i}}{\log{R}}\Bigr)=\Pi_g (\log{R})^r\idotsint\limits_{t_1,\dots,t_r\ge 0} \Phi(\sum_{i=1}^r t_i)\prod_{i=1}^r G(t_i)dt_i\\
&+O\Bigl(r \Omega_G\Pi_g(\log{R})^{r-1}\log\log{R}\idotsint\limits_{t_1,\dots,t_r\ge 0}\prod_{i=1}^r G(t_i)dt_i\Bigr),
\end{align*}
where
\[\Pi_g=\prod_{p}\Bigl(1+\frac{n(p)}{g(p)}\Bigr)\Bigl(1-\frac{1}{p}\Bigr)^r,\qquad n(p)=\#\{i\in\{1,\dots,r\}:p\nmid W_i\}.\]
\end{lmm}
\begin{proof}
We let $\Sigma$ denote the sum in the statement of the lemma. We estimate the $\Sigma$ by applying Lemma \ref{lmm:PartialSummation} $r$ times to each variable $e_1,\dots,e_r$ in turn. We use induction to establish that, having applied the lemma $j$ times, we obtain the estimate
\begin{align}
\Sigma=c_j(\log{R})^j\sum_{\substack{e_{j+1},\dots,e_r\\ (e_i,W_i)=1\forall i}}\frac{\mu(e_{j+1}\dots e_r)^2}{g_j(e_{j+1}\dots e_k)}\prod_{i=j+1}^r G(u_i)\idotsint\limits_{t_1,\dots,t_j\ge 0}\Phi\Bigl(\sum_{i=1}^jt_i+\sum_{i=j+1}^ru_i\Bigr)\prod_{i=1}^j G(t_i)dt_i\nonumber\\
+c_j(\log{R})^j\Bigl(\int_0^\infty G(t)dt\Bigr)^j\Biggl(\sum_{\ell =1}^j\binom{j}{\ell}O\Bigl(\frac{\Omega_G\log\log{R}}{\log{R}}\Bigr)^\ell\Biggr)\sum_{\substack{e_{j+1},\dots,e_r\\ (e_i,W_i)=1\forall i}}\frac{\mu(e_{j+1}\dots e_r)^2}{g_j(e_{j+1}\dots e_k)}\prod_{i=j+1}^r G(u_i),\label{eq:JthExpression}
\end{align}
where
\begin{align}
u_i&=\frac{\log{e_i}}{\log{R}},\qquad &n_j(p)&=\#\{i\in\{1,\dots,j\}:p\nmid W_i\},\\
g_j(d)&=\prod_{p|d}(g(p)+n_j(p)), &c_j&=\prod_p\Bigl(1+\frac{n_j(p)}{g(p)}\Bigr)\Bigl(1-\frac{1}{p}\Bigr)^j.\nonumber
\end{align}
We see that \eqref{eq:JthExpression} clearly holds when $j=0$. We now assume that \eqref{eq:JthExpression} holds for some $j<r$, and apply Lemma \ref{lmm:PartialSummation} to the sum over $e_{j+1}$. In the notation of Lemma \ref{lmm:PartialSummation}, we have
\begin{align}
\gamma(p)&=\begin{cases}
0,\qquad &p|W_{j+1}\prod_{i=j+2}^r e_i,\\
p(1+n_{j}(p)+g(p))^{-1}=1+O(k/p),&p\nmid W_{j+1}\prod_{i=j+2}^r e_i.
\end{cases}
\end{align}
Since $W_{j+1}$ is a multiple of all primes $p\le 2k^2$ (by assumption of the lemma), we see that we can take $A_1$ and $A_2$ to be fixed constants (independent of $j,k,r,x$) provided $k$ is sufficiently large. With this choice of $\gamma(p)$, we see that
\begin{align}
L&\ll 1+\sum_{p|W_{j+1}\prod_{i=j+2}^r e_i}\frac{\log{p}}{p}+\sum_{p>2k^2}\frac{k\log{p}}{p^2}\ll \log\log{R}.
\end{align}
Here we used the fact the first sum is over prime divisors of an integer which is $\ll R^{O(k^2)}$, and this sum is largest when all the prime divisors are smallest, and that $k\ll (\log{R})^{1/5}\ll \log{R}$.

We apply Lemma \ref{lmm:PartialSummation} to the main term with the smooth function $G_1$, and to the error term with the smooth function $G_2$ defined by
\begin{align}
G_1(t)&=\idotsint\limits_{t_1,\dots,t_{j}\ge 0}G(t)\Phi\Bigl(\sum_{i=1}^{j}t_i+t+\sum_{i=j+2}^ru_i\Bigr)\Bigl(\prod_{i=1}^{j}G(t_i)dt_i\Bigr)\Bigl(\prod_{i=j+2}^r G(u_i)\Bigr),\\
G_2(t)&=G(t)\Bigl(\int_{t\ge 0}G(t')dt'\Bigr)^j\Bigl(\prod_{i=j+2}^r G(u_i)\Bigr),
\end{align}
where we recall $u_i=(\log{e_i})/\log{R}$ for $i>j+1$. With this choice, we see that from the bounds on $\Phi, G$ given in the lemma, we have
\begin{align}
\sup_{t\in[0,1]}(|G_{1}(t)|+|G_{1}'(t)|+|G_2(t)|+|G_2'(t)|)&\ll \Omega_G\Bigl(\int_{t\ge 0}G(t)dt\Bigr)^{j+1}\Bigl(\prod_{i=j+2}^r G(u_i)\Bigr)\nonumber\\
&=\Omega_G\int_{t\ge 0}G_2(t)dt.
\end{align}
Thus Lemma \ref{lmm:PartialSummation} gives
\begin{align}
\sum_{\substack{e_{j+1}\\ (e_{j+1},W_{j+1}\prod_{i=j+2}^re_i)=1}}\frac{\mu^2(e_{j+1})}{g_j(e_{j+1})}G_{1}\Bigl(\frac{\log{e_{j+1}}}{\log{R}}\Bigr)=\log{R}\prod_{p}\Bigl(1-\frac{\gamma(p)}{p}\Bigr)^{-1}\Bigl(1-\frac{1}{p}\Bigr)\int_0^\infty G_{1}(t)dt\nonumber\\
+O\Bigl(\Omega_G\log\log{R}\prod_{p}\Bigl(1-\frac{\gamma(p)}{p}\Bigr)^{-1}\Bigl(1-\frac{1}{p}\Bigr)\int_0^\infty G_2(t) dt\Bigr),\label{eq:PerformedSummation}
\end{align}
and we obtain the same expression when summing with $G_2$ instead of $G_1$, except $\int_0^\infty G_1(t)dt$ is replaced by $\int_0^\infty G_2(t)dt$ in the main term. The implied constant in the error term is independent of $j$. We note that
\begin{align}
c_j\prod_p\Bigl(1-\frac{\gamma(p)}{p}\Bigr)^{-1}\Bigl(1-\frac{1}{p}\Bigr)&=\prod_{p| W_{j+1}}\Bigl(1+\frac{n_j(p)}{g(p)}\Bigr)\Bigl(1-\frac{1}{p}\Bigr)^{j+1}\prod_{p\nmid W_{j+1}}\Bigl(1+\frac{n_j(p)+1}{g(p)}\Bigr)\Bigl(1-\frac{1}{p}\Bigr)^{j+1}\nonumber\\
&\times \prod_{\substack{p|e_{j+2}\dots e_{r}\\ p\nmid W_{j+1}}}\Bigl(\frac{n_j(p)+g(p)}{n_j(p)+g(p)+1}\Bigr)\nonumber\\
&=\frac{c_{j+1}g_j(e_{j+2}\dots e_{r})}{g_{j+1}(e_{j+2}\dots e_{r})}.\label{eq:PartialSingularSeries}
\end{align}
Therefore substituting \eqref{eq:PerformedSummation} and \eqref{eq:PartialSingularSeries} into \eqref{eq:JthExpression} gives the result for $j+1$. We conclude that \eqref{eq:JthExpression} holds for all $j\le r$.

Finally, let $\varepsilon=(\Omega_G\log\log{R})/\log{R}$. By assumption of the lemma, we have $\varepsilon=o(1/r)$. We see the sum over $\ell$ in \eqref{eq:JthExpression} is $(1+O(\varepsilon))^j-1=O(j\varepsilon)$ where, by our bound on $\varepsilon$, the implied constant is independent of $j\le r$. Substituting this into \eqref{eq:JthExpression} with $j=r$ gives the result.
\end{proof}
\begin{lmm}\label{lmm:LambdaSize}
Let $k\le (\log{x})^{1/5}$ be sufficiently large in terms of $\theta$. Then we have
\begin{align*}
&(i)\qquad |\lambda_\mathbf{d}|\ll k^{-k}(\log{R})^k,\\
&(ii)\qquad w_n\ll k^{-2k}(\log{x})^{2k}\prod_{i=1}^k\prod_{p|L_i(n),p\nmid B}4,\\
&(iii)\qquad w_n\ll R^{2+o(1)}.\end{align*}
\end{lmm}
\begin{proof}
Substituting in our choice of $y_{\mathbf{r}}$, we have for $\mathbf{d}\in\mathcal{D}_k$
\begin{equation}
|\lambda_{\mathbf{d}}|=d\sum_{\mathbf{d}|\mathbf{r}}\frac{y_{\mathbf{r}}}{\phi_\omega(r)}=\frac{d W^k B^k \mathfrak{S}_{WB}(\mathcal{L}) }{\phi_\omega(d)\phi(W B)^k }\sum_{\mathbf{d}|\mathbf{r}\in\mathcal{D}_k}\frac{1}{\phi_\omega(r/d)}F\Bigl(\frac{\log{r_1}}{\log{R}},\dots,\frac{\log{r_k}}{\log{R}}\Bigr).\label{eq:LambdaSizeExpression}
\end{equation}
We obtain an upper bound for \eqref{eq:LambdaSizeExpression} by replacing the $\log{r_i}/\log{R}$ in the argument of $F$ with $\sigma_i=(\log{r_i/d_i})/\log{R}$, since $F$ is decreasing in each argument.

We now estimate the sum using Lemma \ref{lmm:MultipleSummation}. We see from \eqref{eq:FDef} that $F$ is of the form $\Phi(\sum_{i=1}^kt_i)\prod_{i=1}^kG(t_i)$, and we have a bound on $G,\Phi$ which corresponds to $\Omega_G=O(kT_k)$ (where $T_k=k\log{k}$ is the constant given by \eqref{eq:FDef}). Since $k\le(\log{x})^{1/5}$, we see that $k^2T_k=o(\log\log{R}/\log{R})$. Finally, we note that the condition $\mathbf{r}\in\mathcal{D}_k$ forces $(r_j,dW_j)=1$ for integers $W_1,\dots,W_j\le x^{O(k)}$ which are all a multiple of $WB$. Thus we can apply Lemma \ref{lmm:MultipleSummation}, which gives
\begin{align}
\sum_{\mathbf{d}|\mathbf{r}\in\mathcal{D}_k}\frac{F(\sigma_1,\dots,\sigma_k)}{\phi_\omega(r/d)}\le \frac{\phi(WB)^k}{W^kB^k}\prod_{p\nmid WB}\Bigl(1+\frac{\omega(p)}{p-\omega(p)}\Bigr)\Bigl(1-\frac{1}{p}\Bigr)^k\idotsint\limits_{t_1,\dots,t_k\ge 0}H(t_1,\dots,t_k)dt_1\dots dt_k,\label{eq:LambdaSummation}\end{align}
where
\begin{align}
H(t_1,\dots,t_k)&=F(t_1,\dots,t_k)+O\Bigl(\frac{k^2 T_k \log\log{R}}{\log{R}}F_1(t_1,\dots,t_k)\Bigr).
\end{align}
Substituting \eqref{eq:LambdaSummation} into \eqref{eq:LambdaSizeExpression}, noting that the singular series cancel and that $H\le (1+o(1))F_1$, we have
\begin{align}
|\lambda_{\mathbf{d}}|&\le (1+o(1))(\log{R})^k\idotsint\limits_{t_1,\dots,t_k\ge 0}F_1(t_1,\dots,t_k)dt_1\dots dt_k\nonumber\\
&\ll (\log{R})^k\Bigl(\int_0^{U_k}\frac{dt}{1+T_kt}\Bigr)^k\ll \Bigl(\frac{\log{R}}{k}\Bigr)^k.
\end{align}
This gives the first claim. The second claim now follows from this bound and the definition \eqref{eq:WnDef} of $w_n$, recalling that $\lambda_{\mathbf{d}}=0$ unless $d_1,\dots, d_k$ are all squarefree and coprime to $B$. Finally, for the third claim, the fact that $\lambda_{\mathbf{d}}$ is supported on $d=d_1\cdots d_k<R$ gives
\begin{align}
w_n&\ll \frac{(\log{R})^{2k}}{k^{2k}}\Bigl(\sum_{d_1\cdots d_k<R}1\Bigr)^2\ll R^{2+o(1)}\Bigl(\sum_{d_1\cdots d_k<R}\frac{1}{d_1\dots d_k}\Bigr)^2\ll R^{2+o(1)}.\qedhere
\end{align}
\end{proof}
We will eventually be interested in the quantities $I_k$, $J_k$ considered in the following lemma.
\begin{lmm}\label{lmm:IkJk}
Given a square-integrable function $G:\mathbb{R}^k\rightarrow\mathbb{R}$, let
\[I_k(G)=\int_0^\infty\dotsi\int_0^\infty G^2dt_1\dots dt_k,\qquad J_k(G)=\int_0^\infty \dots \int_0^\infty \Bigl(\int_0^\infty G \, dt_k\Bigr)^2dt_1\dots dt_{k-1}.\]
Let $F$, $F_1$, $F_2$ be as given by \eqref{eq:FDef} and \eqref{eq:F1F2Def}. Then
\begin{align*}
\frac{1}{(2k\log{k})^k}\ll I_k(F)\le \frac{1}{(k\log{k})^k},\qquad &\frac{\log{k}}{k}\ll \frac{J_k(F)}{I_k(F)}\ll \frac{\log{k}}{k},\\
I_k(F)\le I_k(F_1)\le I_k( F_2 )/k^2\ll I_k(F),\qquad &J_k(F)\le J_k(F_1)\le J_k( F_2 )/k^2\ll J_k(F).
\end{align*}
\end{lmm}
\begin{proof}
A minor adaption of the argument of \cite[Section 7]{Maynard} to account for the slightly different definition of $F$ shows
\begin{align}
J_k(F)&\ge \idotsint\limits_{\sum_{i=1}^{k-1}t_i<9/10-U_k}\Bigl(\int_0^\infty F_1 dt_k\Bigr)^2dt_1\dots dt_{k-1}\nonumber\\
&\gg \idotsint\limits_{t_1,\dots,t_{k-1}\ge 0}\Bigl(\int_0^\infty F_1 dt_k\Bigr)^2dt_1\dots dt_{k-1}=J_k(F_1).
\end{align}
Applying the same concentration of measure argument to $I_k(F)$ yields
\begin{align}
I_k(F)&\ge \idotsint\limits_{\sum_{i=1}^{k}t_i<9/10} F_1^2 dt_1\dots dt_{k}\gg \idotsint\limits_{t_1,\dots,t_{k}\ge 0}F_1^2 dt_1\dots dt_{k}=I_k(F_1).
\end{align}
We also have the trivial bounds $I_k(F)\le I_k(F_1)\le k^{-2}I_k(F_2)$ and $J_k(F)\le J_k(F_1)\le k^{-2}J_k(F_2)$.
For our choice of $\psi$, $T_k$, $U_k$ from \eqref{eq:FDef}, we see that
\begin{align}
\int_0^\infty\frac{\psi(t/U_k)dt}{1+T_kt}&= \int_0^{9U_k/10}\frac{dt}{1+T_kt}+O\Bigl(\int_{9U_k/10}^{U_k}\frac{dt}{1+T_kt}\Bigr)=\frac{1}{2k}+O\Bigl(\frac{1}{k\log{k}}\Bigr),\\
	\int_0^\infty\frac{\psi(t/U_k)^2dt}{(1+T_kt)^2}&= \int_0^{9U_k/10}\frac{dt}{(1+T_kt)^2}+O\Bigl(\int_{9U_k/10}^{U_k}\frac{dt}{(1+T_kt)^2}\Bigr)=\frac{1+O(k^{-1/2})}{k\log{k}},\\
\int_0^\infty \frac{\psi(t/2)^2 dt}{(1+T_kt)^2}&=\int_0^{9/5}\frac{dt}{(1+T_kt)^2}+O\Bigl(\int_{9/5}^{2}\frac{dt}{(1+T_kt)^2}\Bigr)=\frac{1+O(k^{-1})}{k\log{k}}.\end{align}
From these bounds it follows immediately that $k^{-2}J_k(F_2)\ll J_k(F_1)$ and $k^{-2}I_k(F_2)\ll I_k(F_1)$ and
\begin{equation}
\frac{J_k(F_1)}{I_k(F_1)}=\frac{\log{k}}{4k}\Bigl(1+O\Bigl(\frac{1}{\log{k}}\Bigr)\Bigr).
\end{equation}
Combining these statements gives the bounds of the Lemma.
\end{proof}
\section{Proof of Propositions}
We see that Lemmas \ref{lmm:SingularSeries}, \ref{lmm:LambdaSize} and \ref{lmm:IkJk} verify the claims at the end of Proposition \ref{prpstn:MainProp} for $w_n$ given by \eqref{eq:WnDef}. We are therefore left to establish the four main claims of Proposition \ref{prpstn:MainProp}, which we now do in turn. To obtain results with the desired uniformity in $k$, we need to perform calculations in a slightly different manner to the corresponding ones in \cite{Maynard}.
\begin{prpstn}\label{prpstn:S1}
Let $w_n$ be as described in Section \ref{sctn:InitialConsiderations}. Then we have
\[\sum_{n\in\mathcal{A}(x)}w_n=\Bigl(1+O\Bigl(\frac{1}{(\log{x})^{1/10}}\Bigr)\Bigr)\frac{B^k}{\phi(B)^k}\mathfrak{S}_B(\mathcal{L})\#\mathcal{A}(x)(\log{R})^k I_k(F).\]
The implied constant depends only on $\theta,\alpha$ and the implied constants from \eqref{eq:Bdd1} and \eqref{eq:Bdd2}.
\end{prpstn}
\begin{proof}
We recall $W=\prod_{p\le 2k^2,p\nmid B}p<\exp((\log{x})^{2/5})$, and consider the summation over $n$ in the residue class $v_0\pmod{W}$. If $(\prod_{i=1}^kL_i(v_0),W)\ne 1$ then we have $w_n=0$, and so we restrict our attention to $v_0$ with $(\prod_{i=1}^kL_i(v_0),W)= 1$. We substitute the definition \eqref{eq:WnDef} of $w_n$, expand the square and swap order of summation. This gives
\begin{align}
\sum_{\substack{n\in\mathcal{A}(x)\\n\equiv v_0 \pmod{W}}}w_n=\sum_{\mathbf{d},\mathbf{e}\in\mathcal{D}_k}\lambda_{\mathbf{d}}\lambda_{\mathbf{e}}\sum_{\substack{n\in\mathcal{A}(x)\\ n\equiv v_0 \pmod{W}\\
[d_i,e_i]|L_i(n)\forall i}}1. 
\label{eq:S1FirstSum}
\end{align}
By our choice of support of the $\lambda_{\mathbf{d}}$, there is no contribution unless $(d_ie_i,d_je_j)=1$ for all $i\ne j$. In this case, given $\mathbf{d},\mathbf{e}\in\mathcal{D}_k$ (so in particular $(d_je_j,a_jW)=1$ for $1\le j\le k$), we can combine the congruence conditions by the Chinese remainder theorem, and see that the inner sum is $\mathcal{A}(x;q,a)$ for some $a$ and for $q=W[\mathbf{d},\mathbf{e}]$. We let $E_q^{(1)}=\max_{a}|\#\mathcal{A}(x;q,a)-\#\mathcal{A}(x)/q|$, and substitute $\#\mathcal{A}(x;q,a)=\#\mathcal{A}(x)/q+O(E^{(1)}_q)$ into \eqref{eq:S1FirstSum}. 

We first show the contribution from the errors $E_q^{(1)}$ are small. There are $O(\tau_{3k}(q))$ ways of writing $q=W[\mathbf{d},\mathbf{e}]$ and all such $q$ are square-free, coprime to $B$ and less than $R^2W<x^\theta$ (since $\lambda_{\mathbf{d}}$ is supported on $d<R\le x^{\theta/3}$). Since $|\lambda_\mathbf{d}|\ll (\log{x})^{k}$ by Lemma \ref{lmm:LambdaSize}, these contribute
\begin{align}
\sum_{\mathbf{d},\mathbf{e}\in\mathcal{D}_k}|\lambda_{\mathbf{d}}\lambda_{\mathbf{e}}|E^{(1)}_q&\ll (\log{x})^{2k}\sum_{q<R^2W, (q,B)=1}\mu^2(q)\tau_{3k}(q)E^{(1)}_q\nonumber\\
&\ll (\log{x})^{2k}\Bigl(\sum_{q<R^2W, (q,B)=1}\mu^2(q)\tau_{3k}(q)^2E_q^{(1)}\Bigr)^{1/2}\Bigl(\sum_{q<R^2W, (q,B)=1}\mu^2(q)E_q^{(1)}\Bigr)^{1/2}.
\end{align}
We apply Hypothesis \ref{hypthss:Weak} to estimate these terms. Using $E_q^{(1)}\ll \#\mathcal{A}(x)/q$ for the first sum, and the average of $E_q^{(1)}$ for the second sum, we see the contribution is
\begin{align}
\ll (\log{x})^{2k}\Bigl(\#\mathcal{A}(x)\sum_{q<R}\frac{\tau_{9k^2}(q)}{q}\Bigr)^{1/2}\Bigl(\frac{\#\mathcal{A}(x)}{(\log{x})^{100k^2}}\Bigr)^{1/2}&\ll\frac{\#\mathcal{A}(x)}{W(\log{x})^{2k^2}}.
\end{align}
By Lemma \ref{lmm:SingularSeries} and Lemma \ref{lmm:IkJk}, we see that this is $o(\#\mathcal{A}(x)\mathfrak{S}_B(\mathcal{L})I_k(F)/W)$, and so will be negligible compared with our main term.

We now consider the main term. We substitute our expression \eqref{eq:YDef} for $\lambda_{\mathbf{d}}$ in terms of $y_{\mathbf{r}}$ to give
\begin{equation}
\frac{\#\mathcal{A}(x)}{W}\sideset{}{'}\sum_{\mathbf{d},\mathbf{e}\in\mathcal{D}_k}\frac{\lambda_{\mathbf{d}}\lambda_{\mathbf{e}}}{[\mathbf{d},\mathbf{e}]}=\frac{\#\mathcal{A}(x)}{W}\sum_{\mathbf{r},\mathbf{s}\in\mathcal{D}_k}\frac{y_{\mathbf{r}}y_{\mathbf{s}}}{\phi_\omega(r)\phi_\omega(s)}
\sideset{}{'}\sum_{\mathbf{d}|\mathbf{r},\mathbf{e}|\mathbf{s}}\frac{\mu(d)\mu(e) d e}{[\mathbf{d},\mathbf{e}]},
\label{eq:S1YrYsSum}
\end{equation}
where $\sum'$ indicates the restriction that $(d_ie_i,d_je_j)=1$ for all $i\ne j$. By multiplicativity, we can write the inner sum as $\prod_{p| rs} S_p(\mathbf{r},\mathbf{s})$,
where, for $\mathbf{r},\mathbf{s}$ such that $p|rs$ and $y_{\mathbf{r}}y_{\mathbf{s}}\ne 0$, we have
\begin{equation}
S_p(\mathbf{r},\mathbf{s})=\sideset{}{'}\sum_{\substack{\mathbf{d}|\mathbf{r},\mathbf{e}|\mathbf{s}\\ d_i,e_i|p\forall i}}\frac{\mu(d) \mu(e) d e}{ [\mathbf{d},\mathbf{e}] }
=
\begin{cases}
p-1,\qquad &p|(\mathbf{r},\mathbf{s}),\\
-1, & p|r,p|s, p\nmid (\mathbf{r},\mathbf{s}),\\
0, &(p|r\text{ and }p\nmid s)\text{ or }(p|s\text{ and }p\nmid r).
\end{cases}
\label{eq:SpDef}
\end{equation}
(We remind the reader that in our notation, $r=\prod_{i=1}^kr_i$ and that $(\mathbf{r},\mathbf{s})=\prod_{i=1}^k(r_i,s_i)$, and similarly for $e,d,s$.)

Since $\prod_{p|rs}S_p(\mathbf{r},\mathbf{s})=0$ if there is a prime $p$ which divides one of $r,s$ but not the other, we can restrict to $r=s$. We let $A=A(\mathbf{r},\mathbf{s})=r/(\mathbf{r},\mathbf{s})$ be the product of primes dividing $r$ but not $(\mathbf{r},\mathbf{s})$, so that $\prod_{p|rs}S_p(\mathbf{r},\mathbf{s})=\mu(A)\phi(r)/\phi(A)$. Given a choice of $\mathbf{r}\in\mathcal{D}_k$ and $A|r$, for each prime $p|A$ there are $\omega(p)-1$ possible choices of which components of $\mathbf{s}$ can be a multiple of $p$ (since there are $\omega(p)$ indices $j$ for which $p\nmid W_j$, but for one of these we have $p|r_j$), and so $\prod_{p|A}(\omega(p)-1)$ choices of $\mathbf{s}$. By Lemma \ref{lmm:YDifference}, for each such choice we have
\begin{equation}y_{\mathbf{s}}=y_{\mathbf{r}}+O\Bigl(T_k (Y_\mathbf{r}+Y_{\mathbf{s}}) \frac{\log{A}}{\log{R}}\Bigr).\end{equation}
Thus our main term becomes
\begin{equation}\frac{\#\mathcal{A}(x)}{W}\sum_{\mathbf{r}\in\mathcal{D}_k}\frac{y_{\mathbf{r}}\phi(r)}{\phi_\omega(r)^2}\sum_{A|r}\Bigl(\prod_{p|A}\frac{-(\omega(p)-1)}{p-1}\Bigr)\Bigl(y_{\mathbf{r}}+O\Bigl(T_k (Y_\mathbf{r}+Y_{\mathbf{s}})\frac{\log{A}}{\log{R}}\Bigr)\Bigr).\end{equation}
Since $y_\mathbf{r}\le Y_\mathbf{r}/k$ and $Y_{\mathbf{r}}Y_{\mathbf{s}}\ll Y_{\mathbf{r}}^2+Y_{\mathbf{s}}^2$, the contribution of the error here is
\begin{align}
&\ll \frac{T_k}{k} \frac{\#\mathcal{A}(x)}{W}
\sum_{\substack{\mathbf{r}\in\mathcal{D}_k}}\frac{\phi(r)Y_{\mathbf{r}}^2}{\phi_{\omega}(r)^2}
\sum_{A|r}\frac{\omega(A)}{\phi(A)}
\sum_{p|A}\frac{\log{p}}{\log{R}}
\nonumber\\
&\ll \frac{T_k}{k} \frac{\#\mathcal{A}(x)}{W\log{R}}
\sum_{p>2k^2}\log{p}\sum_{\substack{(A,WB)=1\\ p|A}}\frac{\omega(A)}{\phi(A)}\sum_{\substack{\mathbf{r}\in\mathcal{D}_k\\ A|r}}\frac{\phi(r)Y_{\mathbf{r}}^2}{\phi_\omega(r)^2}.
\end{align}
We let $\mathbf{r}'$ be the vector formed by removing from $\mathbf{r}$ any factors of $A$, so $r_i'=r_i/(r_i,A)$. Since $Y_{\mathbf{r}}$ is decreasing, we have $Y_{\mathbf{r}'}\ge Y_{\mathbf{r}}$. Given $\mathbf{r}'$, there are $O(\omega(A))$ possible choices of $\mathbf{r}$. Thus, swapping the summation to $\mathbf{r}'$, and letting $A=pA'$, we obtain the bound
\begin{align}
&\ll \frac{T_k}{k} \frac{\#\mathcal{A}(x)}{W\log{R}}
\Bigl(\sum_{p>2k^2}\frac{\omega(p)^2\log{p}}{\phi_{\omega}(p)^2}\Bigr)
\Bigl(\sum_{(A',WB)=1}\frac{\omega(A')^2}{\phi_{\omega}(A')^2}\Bigr)
\Bigl(\sum_{\substack{\mathbf{r}'\in\mathcal{D}_k}}\frac{\phi(r')Y_{\mathbf{r}'}^2}{\phi_\omega(r')^2}\Bigr).\label{eq:S1ErrorSum}
\end{align}
The first two terms in parentheses can both be seen to be $O(1)$, since all prime factors are greater than $2k^2$. We estimate the final term by Lemma \ref{lmm:MultipleSummation} (taking $\Omega_G=O(T_k^2)$). This gives a bound for \eqref{eq:S1ErrorSum} of size
\begin{align}
&\ll \frac{T_k W^{k-1} B^k(\log{R})^{k-1}\mathfrak{S}_{WB}(\mathcal{L})^2 \#\mathcal{A}(x)}{k\phi(WB)^k}\prod_{p\nmid WB}\Bigl(1+\frac{\omega(p)(p-1)}{(p-\omega(p))^2}\Bigr)\Bigl(1-\frac{1}{p}\Bigr)^k I_k( F_2 ).\label{eq:S1ErrorTerm1}
\end{align}
We note that
\begin{equation}
\prod_{p\nmid WB}\Bigl(1+\frac{\omega(p)(p-1)}{(p-\omega(p))^2}\Bigr)\Bigl(1-\frac{1}{p}\Bigr)^k=\prod_{p\nmid WB}\Bigl(1+\frac{\omega(p)}{p-\omega(p)}\Bigr)\Bigl(1+O\Bigl(\frac{k^2}{p^2}\Bigr)\Bigr)\Bigl(1-\frac{1}{p}\Bigr)^k\ll \mathfrak{S}_{WB}(\mathcal{L})^{-1},\label{eq:S1ProductBound}
\end{equation}
since the product is only over primes $p>2k^2$. Using $I_k(F_2)\ll k^2I_k(F)$ from Lemma \ref{lmm:IkJk}, we see that \eqref{eq:S1ErrorTerm1} is
\begin{equation}
\ll \frac{kT_k W^{k-1}B^k\mathfrak{S}_{WB}(\mathcal{L})\#\mathcal{A}(x)(\log{R})^{k-1}}{\phi(WB)^k}I_k(F).
\label{eq:S1ErrorTerm}
\end{equation}
This is negligible, and can be absorbed into the error term in the statement of the Lemma. We now consider the main term. We have
\begin{equation}\frac{\#\mathcal{A}(x)}{W}\sum_{\mathbf{r}\in\mathcal{D}_k}\frac{y_{\mathbf{r}}^2\phi(r)}{\phi_\omega(r)^2}\sum_{A|r}\prod_{p|A}\frac{-(\omega(p)-1)}{p-1}=\frac{\#\mathcal{A}(x)}{W}\sum_{\mathbf{r}\in\mathcal{D}_k}\frac{y_{\mathbf{r}}^2}{\phi_\omega(r)}.\end{equation}
We estimate the inner sum here by applying Lemma \ref{lmm:MultipleSummation} (again with $\Omega_G=T_k^2$). This gives
\begin{align}
\sum_{\mathbf{r}\in\mathcal{D}_k}\frac{y_{\mathbf{r}}^2}{\phi_\omega(r)}
&=\frac{W^k B^k \mathfrak{S}_{WB}(\mathcal{L})^{2} }{\phi(WB)^k}(\log{R})^k
\prod_{p\nmid WB}\Bigl(1+\frac{\omega(p)}{p-\omega(p)}\Bigr)\Bigl(1-\frac{1}{p}\Bigr)^{k}
I_k(F)\nonumber\\
&+O\Biggl(\frac{W^k B^k\mathfrak{S}_{WB}(\mathcal{L})^{2}}{\phi(WB)^k} (\log{R})^k
\prod_{p\nmid WB}\Bigl(1+\frac{\omega(p)}{p-\omega(p)}\Bigr)\Bigl(1-\frac{1}{p}\Bigr)^{k}
\frac{k T_k^2\log\log{R}}{\log{R}}I_k(F_1)\Biggr)\nonumber \\
&=\frac{W^k B^k \mathfrak{S}_{WB}(\mathcal{L})}{\phi(WB)^k}(\log{R})^k\Bigl(1+O\Bigl(\frac{k T_k^2\log\log{R}}{\log{R}}\Bigr)\Bigr)
I_k(F).
\label{eq:S1Result}
\end{align}
In the last line we have used the fact $I_k(F_1)\ll I_k(F)$ given by Lemma \ref{lmm:IkJk}.
Putting this all together (and recalling $k\le (\log{x})^{1/5}$ and $T_k=k\log{k}$), we have shown that
\begin{equation}
\sum_{\substack{n\in\mathcal{A}(x)\\n\equiv v_0 \pmod{W}}}w_n=\Bigl(1+O\Bigl(\frac{1}{(\log{x})^{1/10}}\Bigr)\Bigr)\frac{W^{k-1} B^k \mathfrak{S}_{WB}(\mathcal{L})\#\mathcal{A}(x)}{\phi(WB)^k}(\log{R})^k
I_k(F).
\end{equation}
Summing this over the $\phi_\omega(W)$ residue classes $v_0\pmod{W}$ such that $(\prod_{i=1}^k L(v_0),W)=1$ then gives the result.
\end{proof}
\begin{prpstn}\label{prpstn:S2}
Let $w_n$ be as described in Section \ref{sctn:InitialConsiderations}. Let $L(n)=a_m n+b_m\in\mathcal{L}$ satisfy $L(n)>R$ for $n\in[x,2x]$ and
\begin{equation}
\sum_{\substack{q\le x^{\theta}\\ (q,B)=1}}\max_{(L(a),q)=1}\Bigl|\#\mathcal{P}_{L,\mathcal{A}}(x;q,a)-\frac{\#\mathcal{P}_{L,\mathcal{A}}(x)}{\phi_L(q)}\Bigr|\ll \frac{\#\mathcal{P}_{L,\mathcal{A}}(x)}{(\log{x})^{100k^2}}.\label{eq:Bdd3}
\end{equation}
Then we have
\begin{align*}
\sum_{n\in\mathcal{A}(x)}\mathbf{1}_\mathcal{P}(L(n))w_n=\Bigl(1+O\Bigl(\frac{1}{(\log{x})^{1/10}}\Bigr)\Bigr)\frac{B^{k-1}}{\phi(B)^{k-1}}\mathfrak{S}_B(\mathcal{L})\#\mathcal{P}_{L,\mathcal{A}}(x)(\log{R})^{k+1} J_{k}(F)\prod_{\substack{p|a_m\\ p\nmid  B}}\frac{p-1}{p}\\
+O\Bigl(\frac{B^k}{\phi(B)^k}\mathfrak{S}_B(\mathcal{L})\#\mathcal{A}(x)(\log{R})^{k-1} I_k(F)\Bigr).
\end{align*}
The implied constants depend only on $\theta,\alpha$ and the implied constants from \eqref{eq:Bdd1}, \eqref{eq:Bdd2} and \eqref{eq:Bdd3}.
\end{prpstn}
\begin{proof}
%
Again we split the sum into residue classes $n\equiv v_0\pmod{W}$. If $(\prod_{i=1}^kL_i(v_0),W)\ne 1$ then we have $w_n=0$, and so we restrict our attention to $v_0$ with $(\prod_{i=1}^kL_i(v_0),W)= 1$. We substitute the definition \eqref{eq:WnDef} of $w_n$, expand the square and swap order of summation. This gives
%
\begin{equation}
\sum_{\substack{n\in\mathcal{A}(x)\\ n\equiv v_0\pmod{W}}}\mathbf{1}_\mathcal{P}(L(n))w_n=\sum_{\mathbf{d},\mathbf{e}\in\mathcal{D}_k}\lambda_{\mathbf{d}}\lambda_{\mathbf{e}}\sum_{\substack{n\in\mathcal{A}(x)\\ n\equiv v_0\pmod{W}\\ [d_i,e_i]|L_i(n)\forall i}}\mathbf{1}_{\mathcal{P}}(L(n)).\label{eq:S2LambdaSum}
\end{equation}
We first show that there is no contribution to our sum from $\lambda_{\mathbf{d}}$ for which $(d_j,a_jb_m-a_mb_j)\ne 1$ for some $j\ne m$. If $p|d_j$ then the inner sum requires that $p|a_jn+b_j$. However, if we also have $p|a_jb_m-b_ja_m$ then this means $p|a_m n+b_m$ (since $(a_j,b_j)=1$ by admissibility of $\mathcal{L}$). Since there is no contribution to our sum unless $L(n)=L_m(n)=a_m n+b_m$ is a prime and since $d_j<R<L(n)$ by the support of $\lambda_\mathbf{d}$ and assumption of the Lemma, we see that there is no contribution from $\lambda_{\mathbf{d}}$ with $(d_j,a_jb_m-a_mb_j)\ne 1$.

Thus we may restrict the support of $\lambda_{\mathbf{d}}$ to $\mathcal{D}_k'$, defined by
\begin{align}
\mathcal{D}_k'&=\{\mathbf{d}\in\mathbb{R}^k:\mu^2(d)=1, (d_j,W_j')=1\forall j\},\qquad W_j'=\prod_{p|W_j(a_jb_m-a_mb_j)}p.
\end{align}
We write $\lambda_{\mathbf{d}}'$ for $\lambda_{\mathbf{d}}$ with this restricted support. We see from this that $p|W_j'/W_j$ iff $p\nmid a_m$ and $j$ was the chosen index for the residue class $-b_m\overline{a_m}\pmod{p}$. (For our fixed set of choices of residue classes given in Section \ref{sctn:InitialConsiderations}.)

We now observe that given $\mathbf{d},\mathbf{e}\in\mathcal{D}_k'$, the inner sum of \eqref{eq:S2LambdaSum} is empty unless $(d_ie_i,d_je_j)=1$ for all $i\ne j$ (since otherwise the divisibility conditions are incompatible). If $(d_ie_i,d_je_j)=1\forall i\ne j$, then we can combine the conditions by the Chinese remainder theorem. This shows the sum is $\#\mathcal{P}_{L,\mathcal{A}}(x;q,a)$ for $q=W[\mathbf{d},\mathbf{e}]$ and some $a$. We note $\#\mathcal{P}_{L,\mathcal{A}}(x;q,a)\ne 0$ iff $(L(a),q)=1$, which occurs iff $d_m=e_m=1$. For such a choice of $\mathbf{d},\mathbf{e}$, we write $\#\mathcal{P}_{L,\mathcal{A}}(x;q,a)=\#\mathcal{P}_{L,\mathcal{A}}(x)/\phi_{L}(q)+O(E_q^{(2)})$, where $E_{q}^{(2)}=\max_{(a,q)=1}|\#\mathcal{P}_{L,\mathcal{A}}(x;q,a)-\#\mathcal{P}_{L,\mathcal{A}}(x)/\phi_{L}(q)|$.

We treat the error term $E_q^{(2)}$ in the same manner as we treated $E_q^{(1)}$ in the proof of Proposition \ref{prpstn:S1}. We note that for all $\mathbf{d},\mathbf{e}\in\mathcal{D}_k'$ we have $(q,B)=1$, allowing us to use Proposition \ref{prpstn:MainProp} for the average of $E_q^{(2)}$. We also note that trivially $\#\mathcal{P}_{L,\mathcal{A}}(x;q,a)\ll \#\mathcal{A}(x;q,a)$, which gives us the bound $E_q^{(2)}\ll \#\mathcal{A}(x)/\phi_{L}(q)$. Thus the same argument shows that these error terms contribute $O(\#\mathcal{A}(x)W^{-1}(\log{x})^{-2k^2})$.

We now consider the main term, given by
\begin{equation}
\frac{\#\mathcal{P}_{L,\mathcal{A}}(x)}{\phi_L(W)}\sideset{}{'}\sum_{\substack{\mathbf{d},\mathbf{e}\in\mathcal{D}_k' \\ d_m=e_m=1}}\frac{\lambda'_{\mathbf{d}}\lambda'_{\mathbf{e}}}{\phi_L([\mathbf{d},\mathbf{e}])},\label{eq:S2MainTerm}
\end{equation}
where we recall $\sum'$ indicates the sum is restricted to $(d_ie_i,d_je_j)=1$ for all $i\ne j$.
We change variables to $y_{\mathbf{r}}^{(m)}$, satisfying
\begin{equation}
y_{\mathbf{r}}^{(m)}=\mu(r)\phi_\omega(r)\sum_{\mathbf{r}|\mathbf{d},d_m=1}\frac{\lambda'_{\mathbf{d}}}{\phi_L(d)},\qquad
\lambda_{\mathbf{d}}'=\mu(d)\phi_L(d)\sum_{\mathbf{d}|\mathbf{r}}\frac{y^{(m)}_{\mathbf{r}}}{\phi_\omega(r)}.\label{eq:YmDef}
\end{equation}
We see from \eqref{eq:YmDef} that $y_{\mathbf{r}}$ are supported on $\mathbf{r}\in\mathcal{D}_k'$ with $r_m=1$. Substituting our expression \eqref{eq:YmDef} for $\lambda'_{\mathbf{d}}$ into our main term \eqref{eq:S2MainTerm} gives
\begin{equation}
\frac{\#\mathcal{P}_{L,\mathcal{A}}(x)}{\phi_L(W)}\sideset{}{'}\sum_{\substack{\mathbf{d},\mathbf{e}\\ d_m=e_m=1}}\frac{\lambda'_{\mathbf{d}}\lambda'_{\mathbf{e}}}{\phi_L([\mathbf{d},\mathbf{e}])}=\frac{\#\mathcal{P}_{L,\mathcal{A}}(x)}{\phi_L(W)}\sum_{\mathbf{r},\mathbf{s}}\frac{y_{\mathbf{r}}^{(m)} y_{\mathbf{s}}^{(m)}}{\phi_\omega(r)\phi_\omega(s)}\prod_{p|rs}S_p^{(m)}(\mathbf{r},\mathbf{s}),\label{eq:S2MainTerm1}
\end{equation}
where now, if $\mathbf{r}$ and $\mathbf{s}$ are such that $y^{(m)}_{\mathbf{r}}y^{(m)}_{\mathbf{s}}\ne 0$ (so $r_m=s_m=1$) and $p|rs$, we have
\begin{equation}
S_p^{(m)}(\mathbf{r},\mathbf{s})=\sideset{}{'}\sum_{\substack{\mathbf{d}|\mathbf{r},\mathbf{e}|\mathbf{s}\\ d_i,e_i|p\forall i\\d_m=e_m=1}}\frac{\mu(d) \mu(e) \phi_L(d) \phi_L(e)}{ \phi_L([\mathbf{d},\mathbf{e}]) }
=
\begin{cases}
p-2,\qquad &p|(\mathbf{r},\mathbf{s}), p\nmid a_m\\
p-1, &p|(\mathbf{r},\mathbf{s}), p|a_m\\
-1, & p|r,p|s, p\nmid (\mathbf{r},\mathbf{s}),\\
0, &(p|r\text{ and }p\nmid s)\text{ or }(p|s\text{ and }p\nmid r),
\end{cases}
\end{equation}
so again we may restrict to $r=s$. We use the following lemma to relate $y^{(m)}_\mathbf{r}$ to $y_{\mathbf{r}}$. 
\begin{lmm}\label{lmm:YmExpression}
Let $\mathbf{r}\in\mathcal{D}_k'$ with $r_m=1$, and let $t_i=\log{r_i}/\log{R}$ for $i\ne m$. Then we have
\begin{align*}y_{\mathbf{r}}^{(m)}&=\log{R}\frac{\phi(a_m W B) W^{k-1}B^{k-1}\mathfrak{S}_{WB}(\mathcal{L})}{a_m\phi(W B)^{k}}\int_0^\infty H(t_1,\dots,t_k)dt_m
\end{align*}
where
\begin{align*}
H(t_1,\dots,t_k)=F(t_1,\dots,t_k)+O\Bigl(\frac{T_k(\log\log{R})^2}{\log{R}} F_2 (t_1,\dots,t_k)\Bigr).
\end{align*}
\end{lmm}
We first complete the proof of Proposition \ref{prpstn:S2}, and then establish the lemma. Given $\mathbf{r},\mathbf{s}\in\mathcal{D}_k'$ with $r_m=s_m=1$ and $r=s$, let $A=A(\mathbf{r},\mathbf{s})$ be the product of primes dividing $r$ but not $(\mathbf{r},\mathbf{s})$. Analogously to Lemma \ref{lmm:YDifference}, we have (for $A>1$)
\begin{align}
y_{\mathbf{r}}^{(m)}&=y_{\mathbf{s}}^{(m)}+O\Bigl( \frac{T_k}{k} (Y_{\mathbf{r}}+Y_{\mathbf{s}})\frac{\phi(a_m W B)}{a_m W B}(\log{A}+(\log\log{R})^2)\Bigr)\nonumber\\
&=y_{\mathbf{s}}^{(m)}+O\Bigl(\frac{T_k}{k} (Y_{\mathbf{r}}+Y_{\mathbf{s}})\frac{\phi(a_m W B)}{a_m W B}(\log{A})(\log\log{R})^2\Bigr).
\end{align}
Substituting this into our main term \eqref{eq:S2MainTerm1}, we are left to estimate
\begin{align}
\sum_{\substack{\mathbf{r},\mathbf{s}\in\mathcal{D}'_k\\ r_m=s_m=1\\ r=s}}\frac{y^{(m)}_{\mathbf{r}}}{\phi_\omega(r)^2}\prod_{p|rs}S_p^{(m)}(\mathbf{r},\mathbf{s})\Bigl(y^{(m)}_{\mathbf{r}}+O\Bigl(\frac{T_k}{k}(Y_\mathbf{r}+Y_{\mathbf{s}})\frac{\phi(a_m W B)}{a_m W B}(\log{A})(\log\log{R})\Bigr)\Bigr).\label{eq:S2MainTemp}
\end{align}
We note that for $r=s$ the value of $\prod_{p|rs}S_p^{(m)}(\mathbf{r},\mathbf{s})$ depends only on $r$ and $A$. Substituting this value for $S_p^{(m)}(\mathbf{r},\mathbf{s})$ gives a main term
\begin{equation}
\sum_{\substack{\mathbf{r}\in\mathcal{D}'_k\\ r_m=1}}\frac{(y_{\mathbf{r}}^{(m)})^2}{\phi_\omega(r)^2}\Bigl(\prod_{p|r}(\phi_L(p)-1)\Bigr)\sum_{A|r}\Bigl(\prod_{p|A}\frac{-1}{\phi_L(p)-1}\Bigr)\sum_{\substack{\mathbf{s}\in\mathcal{D}'_k\\ A(\mathbf{r},\mathbf{s})=A}}1,\label{eq:S2MainTerm2}
\end{equation}
and (using $y^{(m)}_{\mathbf{r}}\le Y_{\mathbf{r}}(\log{R})/k$ and $Y_{\mathbf{r}}Y_{\mathbf{s}}\le Y_{\mathbf{r}}^2+Y_{\mathbf{s}}^2$) an error term of size
\begin{align}
\ll \frac{T_k\phi(a_m W B)^2\log{R}}{k^2 a_m^2B^2W^2}\sum_{\substack{\mathbf{r}\in\mathcal{D}'_k\\ r_m=1}}\frac{Y_{\mathbf{r}}^2\prod_{p|r}(\phi_L(p)-1)}{\phi_\omega(r)^2}\sum_{A|r}\frac{\log{A}}{\prod_{p|A}(\phi_L(p)-1)}\sum_{\substack{\mathbf{s}\in\mathcal{D}'_k\\ A(\mathbf{r},\mathbf{s})=A}}(\log\log{R})^2.
\label{eq:S2Error}
\end{align}
We first estimate the inner sum over $\mathbf{s}$ which occurs in both terms. We fix a choice of $\mathbf{r}\in\mathcal{D}'_k$ and $A=A(\mathbf{r},\mathbf{s})$ with $A|r$. For each prime $p|A$, we count how many components of $\mathbf{s}$ can be a multiple of $p$, subject to the constraints that $p\nmid (s_i,r_i)$ and $p\nmid (s_i,W_i')$ for all $i$. If $p|A, p\nmid a_m$ then there are $\omega(p)-2$ possible choices of which component of $\mathbf{s}$ can be a multiple of $p$ (there are $\omega(p)-1$ indices $j\ne m$ for which $p\nmid W_j'$, but for one of these indices $p|r_j$). If $p|A$ and $p|a_m$, then instead there are $\omega(p)-1$ choices (since there are $\omega(p)$ indices $j\ne m$ for which $p\nmid W_j'$, but for one of these indices $p|r_j$). Thus we have
\begin{align}
\sum_{\substack{\mathbf{s}\in\mathcal{D}'_k\\ A(\mathbf{r},\mathbf{s})=A}}1=\prod_{p|A,p\nmid a_m}(\omega(p)-2)\prod_{p|A,p|a_m}(\omega(p)-1).\label{eq:S2sSum}
\end{align}
We now consider the error term \eqref{eq:S2Error}. We follow an analogous argument to that in the proof of Proposition \ref{prpstn:S1}. Substituting our expression \eqref{eq:S2sSum} for the inner sum, and crudely bounding the multiplicative functions gives a bound
\begin{align}
&\ll \frac{T_k\phi(a_m W B)^2(\log{R})(\log\log{R})^2}{k^2 a_m^2B^2W^2}\sum_{(A,WB)=1}\log{A}\frac{\omega(A)}{\phi_\omega(A)^2}\sum_{\substack{\mathbf{r}\in\mathcal{D}'_k\\ r_m=1\\ A|r}}\frac{Y_{\mathbf{r}}^2\phi(r/A)}{\phi_\omega(r/A)^2}.
\end{align}
We let $\mathbf{r}'$ be given by $r'_i=r_i/(r_i,A)$ and see $Y_{\mathbf{r}'}\ge Y_{\mathbf{r}}$. Moreover, we see that there are $O(\omega(A))$ choices of $\mathbf{r}$ given $\mathbf{r}'$. Therefore we obtain the bound
\begin{align}
\sum_{(A,WB)=1}\log{A}\frac{\omega(A)}{\phi_\omega(A)^2}\sum_{\substack{\mathbf{r}\in\mathcal{D}'_k\\ r_m=1\\ A|r}}\frac{Y_{\mathbf{r}}^2\phi(r/A)}{\phi_\omega(r/A)^2}&\ll \Bigl(\sum_{(A,WB)=1}\log{A}\frac{\omega(A)^2}{\phi_\omega(A)^2}\Bigr)\sum_{\substack{\mathbf{r}'\in\mathcal{D}_k\\ r'_m=1}}\frac{Y_{\mathbf{r}'}^2\phi(r')}{\phi_\omega(r')^2}.\label{eq:S2TempBound}
\end{align}
Here dropped the requirement that $(r',A)=1$ for an upper bound. We substitute $\log{A}=\sum_{p|A}\log{p}$, $A=pA'$, and swap the order of summation. This shows the right hand side of \eqref{eq:S2TempBound} is
\begin{align}
&\ll 
\Bigl(\sum_{p>2k^2}\frac{\omega(p)^2\log{p}}{\phi_{\omega}(p)^2}\Bigr)
\Bigl(\sum_{(A',WB)=1}\frac{\omega(A')^2}{\phi_{\omega}(A')^2}\Bigr)
\Bigl(\sum_{\substack{\mathbf{r}'\in\mathcal{D}'_k\\r_m=1}}\frac{\phi(r') Y_{\mathbf{r}'}^2}{\phi_\omega(r')^2}\Bigr).\label{eq:S2Temp2}
\end{align}
The first two sums are seen to be $O(1)$ since they only involve primes $p>2k^2$. The final sum we estimate using Lemma \ref{lmm:MultipleSummation}. This gives a bound for \eqref{eq:S2Temp2} of
\begin{align}
\ll &\frac{ W^{k+1}B^{k+1} \mathfrak{S}_{WB}(\mathcal{L})^2 (\log{R})^{k-1}}{ \phi(WB)^{k+1}}\prod_{p\nmid WB}\Bigl(1+\frac{\omega(p)-1}{p+O(k)}\Bigr)\Bigl(1-\frac{1}{p}\Bigr)^{k-1}\nonumber\\
&\qquad\times \idotsint\limits_{t_1,\dots,t_k\ge 0, t_m=0}F_2(t_1,\dots,t_k)^2\prod_{i\ne m}dt_i.
\end{align}
We see that the product is $O(\mathfrak{S}_{WB}(\mathcal{L})^{-1})$ analogously to \eqref{eq:S1ProductBound}. We also have, from the definition \eqref{eq:F1F2Def} of $F_2$
\begin{equation}
\idotsint\limits_{t_1,\dots,t_k\ge 0, t_m=0}F_2(t_1,\dots,t_k)^2\prod_{i\ne m}dt_i\le k^2\Bigl(\int_0^\infty\frac{\psi(t/U_k)^2dt}{(1+T_kt)^2}\Bigr)^{k-2}\Bigl(\int_0^\infty\frac{\psi(t/2)^2dt}{(1+T_kt)^2}\Bigr)\ll k^2T_k^2 J_k(F).
\end{equation}
Putting this together, the contribution of the error term to \eqref{eq:S2MainTemp} is
\begin{align}
&\ll \frac{ T_k^3\phi(a_m W B)^2 W^{k-1}B^{k-1} \mathfrak{S}_{WB}(\mathcal{L}) (\log{R})^k(\log\log{R})^2}{ a_m^2\phi(WB)^{k+1} }J_k(F),
\label{eq:S2ErrorSum}
\end{align}
which contributes only to the error term in the statement of the Lemma.

We now consider the main term in \eqref{eq:S2MainTemp}, given by \eqref{eq:S2MainTerm2}. Substituting our expression \eqref{eq:S2sSum} for the inner sum, and evaluating the sum over $A$ gives
\begin{align}
\sum_{\substack{\mathbf{r}\in\mathcal{D}_k\\ r_m=1}}\frac{(y_{\mathbf{r}}^{(m)})^2}{\phi_\omega(r)^2}\Bigl(\prod_{p|r}(\phi_L(p)-1)\Bigr)\sum_{A|r}\Bigl(\prod_{p|A}\frac{-1}{\phi_L(p)-1}\Bigr)\sum_{\substack{\mathbf{s}\in\mathcal{D}'_k\\ A(\mathbf{r},\mathbf{s})=A}}1=\sum_{\substack{\mathbf{r}\in\mathcal{D}_k'\\ r_m=1}}\frac{(y_{\mathbf{r}}^{(m)})^2}{\phi_\omega(r)}.
\end{align}
We evaluate this sum using Lemma \ref{lmm:MultipleSummation}. This gives
\begin{align}
\sum_{\substack{\mathbf{r}\in\mathcal{D}_k'\\ r_m=1}}\frac{(y_{\mathbf{r}}^{(m)})^2}{\phi_\omega(r)}
&=\prod_{p\nmid a_m B W}\Bigl(1+\frac{\omega(p)-1}{\phi_\omega(p)}\Bigr)\Bigl(1-\frac{1}{p}\Bigr)^{k-1}
\hspace{-5pt}\prod_{p| a_m, p\nmid W B}\hspace{0pt minus 1fil}\Bigl(1+\frac{\omega(p)}{\phi_\omega(p)}\Bigr)\Bigl(1-\frac{1}{p}\Bigr)^{k-1}\nonumber\\
&\times (\log{R})^{k+1}\frac{\phi(a_m W B)^2 W^{k-1}B^{k-1} \mathfrak{S}_{WB}(\mathcal{L})^2}{a_m^2 \phi(WB)^{k+1}}\Bigl(J_k(H)+O\Bigl(\frac{k T_k^2(\log\log{R})^2}{\log{R}}J_k( F_1 )\Bigr)\Bigr).\label{eq:S2FirstBnd}
\end{align}
By Lemma \ref{lmm:IkJk}, we have $J_k(F_1) \gg J_k(F)$. From the definition of $H$, we have
\begin{align}
J_k(H)=J_k\Bigl(F+O\Bigl(\frac{T_k\log\log{R}}{\log{R}}F_2\Bigr)\Bigr)&=J_k(F)+O\Bigl(\frac{T_k\log\log{R}}{\log{R}}J_k(F_2)\Bigr)\nonumber\\
&=J_k(F)\Bigl(1+O\Bigl(\frac{k^2T_k\log\log{R}}{\log{R}}\Bigr)\Bigr).
\end{align} 
We recall $k\le (\log{x})^{1/5}$ and $T_k=k\log{k}$, so the errors appearing are $o((\log{x})^{-1/10})$. Therefore, simplifying the products in \eqref{eq:S2FirstBnd} gives
\begin{equation}
\sum_{\substack{\mathbf{r}\in\mathcal{D}_k'\\ r_m=1}}\frac{(y_{\mathbf{r}}^{(m)})^2}{\phi_\omega(r)}
=\Bigl(1+\Bigl(\frac{1}{(\log{x})^{1/10}}\Bigr)\Bigr)(\log{R})^{k+1}\frac{W^{k-1}B^{k-1}\mathfrak{S}_{WB}(\mathcal{L})}{\phi(WB)^{k-1}}J_k(F)
\prod_{p|a_m, p\nmid W B}\frac{p-1}{p}.
\end{equation}
Thus, putting everything together, we have
\begin{align}
\sum_{\substack{n\in\mathcal{A}(x)\\ n\equiv v_0\pmod{W}}}&\mathbf{1}_\mathcal{P}(L(n))w_n=\frac{\#\mathcal{P}_{L,\mathcal{A}}(x)}{\phi_L(W)}\sideset{}{'}\sum_{\substack{\mathbf{d},\mathbf{e}\\ d_m=e_m=1}}\frac{\lambda'_{\mathbf{d}}\lambda'_{\mathbf{e}}}{\phi_L([\mathbf{d},\mathbf{e}])}+O\Bigl(\frac{\#\mathcal{A}(x)}{W(\log{x})^{2k^2}}\Bigr)\nonumber\\
&=\Bigl(1+O\Bigl(\frac{1}{(\log{x})^{1/10}}\Bigr)\Bigr)(\log{R})^{k+1}\frac{W^{k-1}B^{k-1}\mathfrak{S}_{WB}(\mathcal{L})\#\mathcal{P}_{L,\mathcal{A}}(x)}{\phi(WB)^{k-1}\phi_L(W)}J_k(F)\prod_{p|a_m,p\nmid  W B}\frac{p-1}{p}\nonumber\\
&\qquad+O\Bigl(\frac{\#\mathcal{A}(x)}{W(\log{x})^{2k^2}}\Bigr).
\end{align}
Summing over the $\phi_\omega(W)$ residue classes $v_0\pmod{W}$ then gives the result (recalling $(W,B)=1$).
\end{proof}
We now return to prove Lemma \ref{lmm:YmExpression}.
\begin{proof}[Proof of Lemma \ref{lmm:YmExpression}]
We substitute our expression \eqref{eq:YDef} for $\lambda_{\mathbf{d}}$ into the definition \eqref{eq:YmDef} of $y^{(m)}_\mathbf{r}$. For $r_m=1$ and $\mathbf{r}\in\mathcal{D}'_k$, we obtain
\begin{align}
y_{\mathbf{r}}^{(m)}
&=\mu(r)\phi_\omega(r)\sum_{\substack{\mathbf{d}\in\mathcal{D}'_k\\ \mathbf{r}|\mathbf{d},d_m=1}}\frac{\lambda_{\mathbf{d}}}{\phi_L(d)}
=\mu(r)\phi_\omega(r)\sum_{\mathbf{r}|\mathbf{e}}\frac{y_\mathbf{e}}{\phi_\omega(e)}
\sum_{\substack{\mathbf{d}\in\mathcal{D}_k'\\ \mathbf{r}|\mathbf{d},\mathbf{d}|\mathbf{e},d_m=1}}\frac{\mu(d)d}{\phi_L(d)}\nonumber\\
&=\frac{r\phi_\omega(r)}{\phi_L(r)}
\sum_{\mathbf{r}|\mathbf{e}}\frac{y_\mathbf{e}}{\phi_\omega(e)}
\prod_{p|e/r}S_p^{' (m)}(\mathbf{e},\mathbf{r}),
\label{eq:YmExpression}
\end{align}
where if $p|e_m$ then $S_p^{' (m)}(\mathbf{e},\mathbf{r})=1$ and if $p|e_j/r_j$ with $j\ne m$, we have
\begin{equation}
S_p^{' (m)}(\mathbf{e},\mathbf{r})=
\sum_{\substack{\mathbf{d}\in\mathcal{D}_k'\\ d_j|(e_j/r_j,p), d_m=1}}\frac{\mu(d)d}{\phi_L(d)}=
\begin{cases}
-1/(p-1), \quad &p\nmid a_mW'_j,\\
0, &p|a_m, p\nmid W_j',\\
1, &p|W'_j/W_j.
\end{cases}
\label{eq:Sm'Expression}
\end{equation}
(Since $\mathbf{e}\in\mathcal{D}_k$, we have $(e_j,W_j)=1$ and so if $p|e_j/r_j$ we only need consider $p\nmid W_j$.)

We let $e_j=r_js_jt_j$ for each $j\ne m$, where $s_j$ is the product of primes dividing $e_j/r_j$ but not $W_j'$, and $t_j$ is the product of primes dividing both $e_j/r_j$ and $W_j'/W_j$. We put $s_m=t_m=1$, and consider $e_m$ separately.

We can restrict to the case when $(s_j,a_m)=1$ for all $j$, since otherwise the product of $S_p^{' (m)}(\mathbf{e},\mathbf{r})$ vanishes. For $\mathbf{e}\in\mathcal{D}_k$, the product in \eqref{eq:YmExpression} is then $\mu(s)/\phi(s)$ by \eqref{eq:Sm'Expression}. Since $(a_m,W_j'/W_j)=1$ for all $j$, we can also restrict to $(t_j,a_m)=1$ for all $j$. (If $p|W_j'/W_j$ then $p|a_mb_j-a_jb_m$, so if $p|a_m$ and $p|W_j'/W_j$, then $p|a_j$ and hence $p|W_j$, meaning $p\nmid W_j'/W_j$).

We let $\mathbf{r}'=(r_1,\dots,r_{m-1},e_m,r_{m+1},\dots,r_k)$. By Lemma \ref{lmm:YDifference}, we have
\begin{equation}
y_{\mathbf{e}}=y_{\mathbf{r}'}+O\Bigl(T_k Y_{\mathbf{r}'}\frac{\log{st}}{\log{R}}\Bigr).
\end{equation}
Substituting this into \eqref{eq:YmExpression} gives
\begin{equation}
y_{\mathbf{r}}^{(m)}=\frac{r}{\phi_L(r)}\sum_{e_m}\frac{y_{\mathbf{r}'}}{\phi_\omega(e_m)}\sum_{\substack{\mathbf{s},\mathbf{t}}}\frac{\mu(s)}{\phi(s)\phi_{\omega}(st)}
+O\Bigl(\frac{T_k r}{\phi_L(r)\log{R}}
\sum_{\substack{e_m\\ \mathbf{r}'\in\mathcal{D}_k}}\frac{Y_{\mathbf{r}'}}{\phi_\omega(e_m)}\sum_{\substack{\mathbf{s},\mathbf{t}}}\frac{\log{st}}{\phi(s)\phi_{\omega}(st)}\Bigr),\label{eq:YmExpression2}
\end{equation}
where the sum is over $\mathbf{s}\in\mathcal{D}_k',\mathbf{t}\in\mathcal{D}_k$ subject to $s_m=t_m=1$, $(s,t)=(st,re_ma_m)=1$, and $t_j|W_j'/W_j$.

We first estimate the error term from \eqref{eq:YmExpression2}. We have $\log{st}\ll s^{1/2}(1+\log{t})$, and we drop the requirement that $(s,t)=1$. The sum over $s$ then factorizes as an Euler product, which can be seen to be $O(1)$ since there are $O(\omega(u))$ choices of $\mathbf{s}$ with $s=u$, and we only consider primes $p>2k^2$. We are summing over square-free $t|\Delta=\prod_{i=1}^k(a_mb_i-a_ib_m)$ with $(t,WBre_ma_m)=1$, and for every such $t$ there is at most one possible $\mathbf{t}$ (since for every prime $p|t$ with $p|W_m$ there is a unique index $j$ such that $p|W_j'/W_j$, and if $p\nmid W_m$ there is no such index). Thus the sum over $\mathbf{t}$ contributes at most
\begin{align}
\sum_{t\in\mathcal{D}_k:t|\Delta}\frac{1+\sum_{p|t}\log{p}}{\phi_\omega(t)}&\ll \Bigl(1+\sum_{p>2k^2:p|\Delta}\frac{\log{p}}{p}\Bigr)\prod_{p>2k^2:p|\Delta}\Bigl(1+\frac{1}{\phi_\omega(p)}\Bigr)\nonumber\\
&\ll (\log\log{\Delta})^2\ll (\log\log{R})^2,
\end{align}
since both sum and product are largest if $\Delta$ is composed of primes $\ll \log{\Delta}$, and $\Delta\ll x^{O(k)}$.

Thus, relaxing the constraint $(e_m,rW_m)=1$ to $(e_m,a_mWBr)=1$, and using Lemma \ref{lmm:MultipleSummation} to estimate the sum over $e_m$, we see the error contributes a total
\begin{align}
&\ll \frac{T_k (\log\log{R})^2}{\log{R}}\frac{r}{\phi_L(r)}\sum_{(e_m,a_mWB r)=1}\frac{Y_{\mathbf{r}'}}{\phi_\omega(e_m)}\nonumber\\
&\ll \frac{T_k(\log\log{R})^2\phi(a_m W B)W^{k-1}B^{k-1}\mathfrak{S}_{WB}(\mathcal{L})}{a_m\phi(WB)^k}\int_0^\infty  F_2 (t_1,\dots,t_k)dt_m.\label{eq:YmError}
\end{align}
We now return to the main term from \eqref{eq:YmExpression2}. We first consider the inner sum, which by multiplicativity we can rewrite as a product
\begin{equation}
\sideset{}{^*}\sum_{\mathbf{s},\mathbf{t}}\frac{\mu(s)}{\phi(s)\phi_{\omega}(st)}=\prod_{p}\sideset{}{^*}\sum_{\substack{\mathbf{s},\mathbf{t}\\ s_i|p, t_i|p\forall i}}\frac{\mu(s)}{\phi(s)\phi_{\omega}(st)},
\end{equation}
where the asterisk indicates that sums are subject to the additional constraints that $s_m=t_m=1$, $(s,t)=1$ and that $(s_i,W_i're_ma_m)=1$, $(t_i,W_ire_ma_m)=1$, and $t_i|W_i'/W_i$ for all $1\le i\le k$. Since the summand only depends on $s$ and $t$ we can evaluate it by counting how many pairs $\mathbf{s},\mathbf{t}$ correspond to a given choice of $s,t$.

If $p|WBre_ma_m$ then no component of $\mathbf{s}$ or $\mathbf{t}$ can be a multiple of $p$. If $p\nmid WBre_ma_m$ then there are $\omega(p)-1$ components of $\mathbf{s}$ which can be a multiple of $p$ (corresponding to the indices for all residue classes chosen mod $p$ except for the index corresponding to $-b_m\overline{a_m}$). If $p\nmid W_mre_m$ then no components of $\mathbf{t}$ can be a multiple of $p$ ($p\nmid W_m$ means $m$ was the chosen index for the residue class $-b_m\overline{a}_m\pmod{p}$, so $p\nmid W_j'/W_j$ for any $j$). If $p|W_m,p\nmid WBra_m$ then exactly one component of $\mathbf{t}$ can be a multiple of $p$ ($t_j$ can be a multiple of $p$ if $j$ was the chosen index for the residue class $-b_m\overline{a}_m\pmod{p}$, and this occurs for the unique $j$ such that $p|W_j'/W_j$). Finally, since $(s,t)=1$, no component of $\mathbf{s}$ can be a multiple of $p$ if $t$ is a multiple of $p$. Putting this together, we obtain (since $(e_m,rW_m)=1$)
\begin{align}
\sideset{}{^*}\sum_{\mathbf{s},\mathbf{t}}\frac{\mu(s)}{\phi(s)\phi_\omega(st)}&=\prod_{p|W_m,p\nmid W  B r a_m}\Bigl(1-\frac{\omega(p)-1}{\phi(p)\phi_\omega(p)}+\frac{1}{\phi_\omega(p)}\Bigr)\prod_{p\nmid W_m re_m}\Bigl(1-\frac{\omega(p)-1}{\phi(p)\phi_\omega(p)}\Bigr)\nonumber\\
&=\prod_{p|W_m,p\nmid W B r a_m}\frac{p}{p-1}\prod_{p\nmid W_m r}\Bigl(\frac{p}{p-1}-\frac{1}{\phi_\omega(p)}\Bigr)\prod_{p|e_m}\Bigl(\frac{p}{p-1}-\frac{1}{\phi_\omega(p)}\Bigr)^{-1}.\label{eq:stSum}
\end{align}
Now, using Lemma \ref{lmm:PartialSummation}, we estimate the summation over $e_m$. This gives
\begin{align}
&\sum_{(e_m,rW_m)=1}\frac{y_{\mathbf{r}'}}{\phi_\omega(e_m)}\prod_{p|e_m}\Bigl(\frac{p}{p-1}-\frac{1}{\phi_\omega(p)}\Bigr)^{-1}\nonumber\\
&=\log{R}\frac{\mathfrak{S}_{WB}(\mathcal{L})W^kB^k}{\phi(WB)^k}\prod_{p|rW_m}\Bigl(1-\frac{1}{p}\Bigr)\prod_{p\nmid rW_m}\Bigl(\frac{p}{p-1}-\frac{1}{\phi_\omega(p)}\Bigr)^{-1}\int_0^\infty H(t_1,\dots, t_k)dt_m,
\label{eq:emSummation}
\end{align}
where we have written $r_i=R^{t_i}$ for $i\ne m$ to simplify notation, and where
\begin{align*}
H(u_1,\dots,u_k)=F(u_1,\dots,u_k)+O\Bigl(\frac{T_k(\log\log{R})^2}{\log{R}} F_2 (u_1,\dots,u_k)\Bigr).
\end{align*}
We have added an additional factor of $\log\log{R}$ into the error term for $H$ so we can absorb \eqref{eq:YmError} into the error term.

Thus, combining \eqref{eq:stSum} and \eqref{eq:emSummation} gives
\begin{align}
&\frac{r}{\phi_L(r)}\sum_{e_m}\frac{y_{\mathbf{r}'}}{\phi_\omega(e_m)}\sum_{\mathbf{s},\mathbf{t}}\frac{\mu(s)}{\phi(s)\phi_\omega(st)}\nonumber\\
&=\log{R}\frac{W^kB^k\mathfrak{S}_{WB}(\mathcal{L})}{\phi(WB)^{k}}\frac{r}{\phi_L(r)}\prod_{p|r}\Bigl(1-\frac{1}{p}\Bigr)\prod_{\substack{p|WBa_m\\ p\nmid r}}\Bigl(1-\frac{1}{p}\Bigr)\int_0^\infty H(t_1,\dots, t_k)dt_m\nonumber\\
&=\log{R}\frac{\phi(a_m W B)W^{k-1}B^{k-1}\mathfrak{S}_{WB}(\mathcal{L})}{a_m\phi(WB)^{k}}\int_0^\infty H(t_1,\dots, t_k)dt_m.\label{eq:YmMain}
\end{align}
Here we have used the fact $\mathbf{r}\in\mathcal{D}_k'$, and so $(r,WB)=1$. Combining \eqref{eq:YmError} and \eqref{eq:YmMain} gives the result.
\end{proof}
\begin{prpstn}\label{prpstn:S3}
Let $w_n$ be as described in Section \ref{sctn:InitialConsiderations}. Given $D,\xi$ satisfying $D\le x^{\alpha}$, and $k(\log\log{x})^2/(\log{x})\le \xi\le \theta/10$ let
\[\mathcal{S}(\xi;D)=\{n\in\mathbb{N}:p|n\implies (p> x^\xi\text{ or }p|D)\}.\]
For $L=a_0n+b_0\notin\mathcal{L}$, with $|a_0|,\,|b_0|\le x^{\alpha}$ and $\Delta_L\ne 0$, we have
\[\sum_{\substack{n\in\mathcal{A}(x)}}\mathbf{1}_{\mathcal{S}(\xi;D)}(L(n))w_n\ll \xi^{-1} \frac{\Delta_L}{\phi(\Delta_L)}\frac{D}{\phi(D)}\frac{B^{k}}{\phi(B)^{k}}\mathfrak{S}_B(\mathcal{L})\#\mathcal{A}(x)(\log{R})^{k-1}I_{k}(F),\]
where
\[\Delta_L=|a_0|\prod_{i=1}^k|a_jb_0-a_0b_j|.\]
The implied constant depends only on $\theta,\alpha$ and the implied constants from \eqref{eq:Bdd1} and \eqref{eq:Bdd2}.
\end{prpstn}
\begin{proof}
We first split the sum into residue classes $v_0$ modulo $V=\prod_{p\le 2k^2}p$ for which $L(v_0)$ is coprime to $\prod_{p\le 2k^2, p\nmid D}p$ and each of the $L_i(v_0)$ are coprime to $W$ (the other residue classes make no contribution because of the support of $w_n$ and $\mathbf{1}_{\mathcal{S}(\xi;D)}$). We use the Selberg sieve upper bound 
\begin{equation}
\mathbf{1}_{\mathcal{S}(\xi;D)}(L(n))\le \tilde{\lambda}_1^{-2}\Bigl(\sum_{\substack{d_0|L(n)\\ d_0<x^\xi\\ (d_0,D)=1}}\tilde{\lambda}_{d_0}\Bigr)^2.
\end{equation}
(This holds for any choice of the values of $\tilde{\lambda}_d\in\mathbb{R}$ with $\tilde{\lambda}_1\ne 0$). For the residue class $v_0\pmod{V}$, this gives
\begin{equation}
\sum_{\substack{n\in \mathcal{A}(x)\\ n\equiv v_0\pmod{V}}}\mathbf{1}_{\mathcal{S}(\xi;D)}(L(n))w_n\le \frac{1}{\tilde{\lambda}_1^2}\sum_{\substack{n\in \mathcal{A}(x)\\ n\equiv v_0\pmod{V}}}\Bigl(\sum_{\substack{d_0|L(n)\\ (d_0,D)=1, d_0<x^\xi}}\tilde{\lambda}_{d_0}\Bigr)^2\Bigl(\sum_{d_i|L_i(n)}\lambda_{\mathbf{d}}\Bigr)^2.\label{eq:S3First}
\end{equation}
We restrict the support of $\tilde{\lambda}_{d_0}$ in a similar way to that of $\lambda_{\mathbf{d}}$. We force $\tilde{\lambda}_{d_0}=0$ if $p|d_0$ for any prime with $p|W_0$ where
\begin{equation}
W_0=D V\Delta_L.
\end{equation}
Similarly, we force $\tilde{\lambda}_{d_0}=0$ if $d_0>x^\xi$. Note that we allow $\tilde{\lambda}_{d_0}\ne 0$ if $(d_0,B)\ne 1$.

We return to \eqref{eq:S3First}. Expanding the squares and swapping the order of summation gives
\begin{equation}
\frac{1}{\tilde{\lambda}_1^2}\sum_{\substack{n\in \mathcal{A}(x)\\ n\equiv v_0\pmod{V}}}\Bigl(\sum_{d_0|L(n)}\tilde{\lambda}_{d_0}\Bigr)^2\Bigl(\sum_{d_i|L_i(n)}\lambda_{\mathbf{d}}\Bigr)^2=\tilde{\lambda}_1^{-2}\sum_{\substack{d_0,e_0\\ (d_0e_0,W_0)=1}}\tilde{\lambda}_{d_0}\tilde{\lambda}_{e_0}\sum_{\mathbf{d},\mathbf{e}\in\mathcal{D}_k}\lambda_{\mathbf{d}}\lambda_{\mathbf{e}}\sum_{\substack{n\in\mathcal{A}(x)\\ n\equiv v_0\pmod{V}\\ [d_i,e_i]|L_i(n)\forall 0\le i\le k}}1.\label{eq:S3Sum}
\end{equation}
We see that by our restrictions on the support of $\lambda_{\mathbf{d}},\tilde{\lambda}_{d_0}$, there is no contribution to \eqref{eq:S3Sum} unless $d_0,e_0,\mathbf{d},\mathbf{e}$ are such that $(d_ie_i,d_je_j)=1$ for all $0\le i\ne j\le k$, and $d,e<R$ and  $d_0,e_0<x^\xi$. (To avoid confusion, we recall that $d=\prod_{i=1}^k d_i$ and $e=\prod_{i=1}^k e_i$). For such values, we can combine the congruence conditions using the Chinese remainder theorem, which shows the inner sum is $\#\mathcal{A}(x;q,a)$ for some $a$ and $q=V\prod_{i=0}^k[d_i,e_i]$. We see that $q<V R^2x^{2\xi}<x^\theta$ since $\xi\le \theta/10$. We substitute $\#\mathcal{A}(x;q,a)=\#\mathcal{A}(x)/q+O(E_q^{(1)})$, and the contribution from $E_q^{(1)}$ can be seen to be negligible by an identical argument to that in the proof of Proposition \eqref{prpstn:S1}. We are therefore left to evaluate
\begin{equation}
\frac{\#\mathcal{A}(x)}{V\tilde{\lambda}_1^2}\sum_{\substack{d_0,e_0\\ (d_0,e_0,W_0)=1}}\frac{\tilde{\lambda}_{d_0}\tilde{\lambda}_{e_0}}{[d_0,e_0]}\sideset{}{'}\sum_{\substack{\mathbf{d},\mathbf{e}\in\mathcal{D}_k\\ (de,d_0e_0)=1}}\frac{\lambda_{\mathbf{d}}\lambda_{\mathbf{e}}}{[\mathbf{d},\mathbf{e}]}.\label{eq:S3MainSum}
\end{equation}
We let $\omega^*$ be the totally multiplicative function defined by 
\begin{equation}
\omega^*(p)=\begin{cases}\#\{1\le n\le p:L(n)\prod_{i=1}^kL_i(n)\equiv 0\pmod{p}\},\qquad &p\nmid B,\\ 
1,&p|B.
\end{cases}
\end{equation}
We note that with this choice, $\omega^*(p)=\omega(p)$ if $p|\Delta_L$ and $p\nmid B$, and $\omega^*(p)=\omega(p)+1$ otherwise. We also define
\begin{equation}
y_{\mathbf{r},r_0}=\mu(r_0r)\phi_{w*}(r_0r)\sum_{\substack{\mathbf{r}|\mathbf{d}\\ r_0|d_0\\ (d_0,d)=1}}\frac{\lambda_{\mathbf{d}}\tilde{\lambda}_{d_0}}{dd_0},\qquad \tilde{y}_{r_0}=\mu(r_0)\phi(r_0)\sum_{r_0|d_0}\frac{\tilde{\lambda}_{d_0}}{d_0}.
\label{eq:S3YDef}
\end{equation}
By Moebius inversion, we see that this definition of $\tilde{y}_{r_0}$ implies that
\begin{equation}
\tilde{\lambda}_{d_0}=\mu(d_0)d_0\sum_{r_0|d_0}\frac{\tilde{y}_{r_0}}{\phi(r_0)},
\end{equation}
For $(r_0,W_0)=1$ and $r_0<x^{\xi}$ we choose
\begin{equation}
\tilde{y}_{r_0}=\frac{W_0}{\phi(W_0)},\label{eq:TildeYDef}
\end{equation} 
and $\tilde{y}_{r_0}=0$ otherwise. This gives rise to a suitable choice of $\tilde{\lambda}_{d_0}$ supported on $d_0<x^\xi$ with $(d_0,W_0)=1$. Since $\xi\gg k(\log\log{x})^2/(\log{x})$, Lemma \ref{lmm:PartialSummation} shows that
\begin{align}
\tilde{\lambda}_1&=\sum_{\substack{r_0<x^\xi\\ (r_0,W_0)=1}}\frac{\tilde{y}_{r_0}\mu^2(r_0)}{\phi(r_0)}=\xi\log{x}+O(\log\log{x})\gg \xi\log{x}.
\end{align}
As in the proof of Proposition \ref{prpstn:S1} (this is exactly the same argument but for $(k+1)$-dimensional vectors instead of $k$-dimensional ones) changing variables using \eqref{eq:S3YDef} shows that
\begin{align}
\sum_{d_0,e_0}\frac{\tilde{\lambda}_{d_0}\tilde{\lambda}_{e_0}}{[d_0,e_0]}\sideset{}{'}\sum_{\substack{\mathbf{d},\mathbf{e}\in\mathcal{D}_k\\ (de,d_0e_0)=1}}\frac{\lambda_{\mathbf{d}}\lambda_{\mathbf{e}}}{[\mathbf{d},\mathbf{e}]}&=\sum_{\mathbf{r},\mathbf{s}\in\mathcal{D}_k, r_0,s_0}\frac{y_{\mathbf{r},r_0}y_{\mathbf{s},s_0}}{\phi_{\omega*}(rr_0)\phi_{\omega*}(ss_0)}\prod_{p|rr_0ss_0}S_p(\mathbf{r},\mathbf{s},r_0,s_0),\label{eq:Yrr0Sub}
\end{align}
where
\begin{align}
S_p(\mathbf{r},\mathbf{s},r_0,s_0)&=\begin{cases}
p-1,\qquad &p|(\mathbf{r},\mathbf{s})(r_0,s_0),\\
-1,& p|rr_0\text{ and }p|ss_0\text{ but }p\nmid (\mathbf{r},\mathbf{s})(r_0,s_0),\\
0,& (p|rr_0\text{ and }p\nmid ss_0)\text{ or }(p|ss_0\text{ and }p\nmid rr_0).
\end{cases}
\end{align}
Thus we may restrict to $rr_0=ss_0$. Using the bound $y_{\mathbf{r},r_0}y_{\mathbf{s},s_0}\ll y_{\mathbf{r},r_0}^2+y_{\mathbf{s},s_0}^2$, we see that (by symmetry) the right hand side of \eqref{eq:Yrr0Sub} may be bounded by
\begin{align}
\sum_{\mathbf{r},r_0}\frac{y_{\mathbf{r},r_0}^2}{\phi_{\omega*}(rr_0)^2}\sum_{\substack{\mathbf{s},s_0\\ ss_0=rr_0}}\prod_{p|rr_0}|S_p(\mathbf{r},\mathbf{s},r_0,s_0)|\le \sum_{\mathbf{r},r_0}y_{\mathbf{r},r_0}^2\prod_{p|rr_0}\frac{p+\omega^*(p)-2}{(p-\omega^*(p))^2}=\sum_{\mathbf{r},r_0}\frac{y_{\mathbf{r},r_0}^2}{\prod_{p|rr_0}(p+O(k))}.\end{align}
To evaluate this, we express $y_{\mathbf{r},r_0}$ in terms of $y_{\mathbf{r}}$ and $\tilde{y}_{r_0}$. Substituting \eqref{eq:YDef} into \eqref{eq:S3YDef}, we find that for $(r_0,rW_0)=1$ and $\mathbf{r}\in\mathcal{D}_k$
\begin{align}
y_{\mathbf{r},r_0}&=\mu(r_0r)\phi_{\omega*}(r_0r)\sum_{r_0|d_0}\mu(d_0)\sum_{d_0|f_0}\frac{\tilde{y}_{f_0}}{\phi(f_0)}\sum_{\substack{\mathbf{r}|\mathbf{d}\nonumber\\ (d,d_0)=1}}\mu(d)\sum_{\mathbf{d}|\mathbf{f}}\frac{y_{\mathbf{f}}}{\phi_\omega(f)}\\
&=\mu(r_0r)\phi_{\omega*}(r_0r)\sum_{\mathbf{r}|\mathbf{f},r_0|f_0}\frac{y_{\mathbf{f}}\tilde{y}_{f_0}}{\phi_\omega(f)\phi(f_0)}\sum_{\substack{r_0|d_0, d_0|f_0\\ \mathbf{r}|\mathbf{d},\mathbf{d}|\mathbf{f}\\ (d,d_0)=1}}\mu(d)\mu(d_0).
\end{align}
The inner sum is 0 unless every prime dividing one of $f,f_0$ but not the other is a divisor of $rr_0$. In this case the sum is $\pm 1$. Thus, using the fact that $y_{\mathbf{r}}\ge y_{\mathbf{f}}$ and $\tilde{y}_{r_0}\ge \tilde{y}_{f_0}$ (since $F$ is decreasing), we have the crude bound
\begin{equation}
y_{\mathbf{r},r_0}\le \phi_{\omega*}(r_0r)y_{\mathbf{r}}\tilde{y}_{r_0}\sum_{\substack{r_0|f_0\\ (f_0,W_0)=1}}\sum_{\substack{\mathbf{r}|\mathbf{f}\in\mathcal{D}_k\\ p\nmid (f,f_0)\Rightarrow p|rr_0}}\frac{\mu^2(f_0)}{\phi(f_0)\phi_\omega(f)}.
\end{equation}
We let $f_0=r_0f_0'g_0$ and $f_i=r_i f_i' g_i$ for $1\le i\le k$, where $f'_i=f_i/(f_i,rr_0)$ is $f_i$ with any factors of $r r_0$ removed, $g_0|r$ and $g_i|r_0$ for $1\le i\le k$. We see the constraint $p\nmid (f,f_0)\Rightarrow p|r r_0$ means that $f_0'=\prod_{i=1}^k f_i'$. Therefore we can bound the double sum above by
\begin{align}
&\frac{1}{\phi(r_0)\phi_\omega(r)}\sum_{\mathbf{f}'\in\mathcal{D}_k}\frac{1}{\phi(f')\phi_\omega(f')}\sum_{\substack{\mathbf{g}\in\mathcal{D}_k\\ g_i|r_0\forall 1\le i\le k}}\frac{1}{\phi_\omega(g)}\sum_{\substack{g_0|r \\ (g_0,W_0)=1}}\frac{1}{\phi(g_0)}\nonumber\\
&=\frac{1}{\phi(r_0)\phi_\omega(r)}\prod_{p\nmid WB}\Bigl(1+\frac{\omega(p)}{(p-1)(p-\omega(p))}\Bigr)\prod_{p|r_0}\Bigl(1+\frac{\omega(p)}{p-\omega(p)}\Bigr)\prod_{p|r, p\nmid W_0}\Bigl(1+\frac{1}{p-1}\Bigr).
\end{align}
The first product is $O(1)$ since it is over primes $p>2k^2$. Thus, simplifying the remaining products, we obtain
\begin{equation}
y_{\mathbf{r},r_0}\ll y_{\mathbf{r}}\tilde{y}_{r_0}\Bigl(\prod_{p|(r,W_0)}\frac{p-\omega^*(p)}{p-\omega(p)}\Bigr)\Bigl(\prod_{p|rr_0, p\nmid W_0}\frac{p(p-\omega^*(p))}{(p-1)(p-\omega(p))}\Bigr)\le y_{\mathbf{r}}\tilde{y}_{r_0}.
\end{equation}
Here we have used the fact that $\omega^*(p)=\omega(p)+1$ if $p\nmid W_0$.

Recalling the definitions \eqref{eq:TildeYDef} and \eqref{eq:YDef} of $\tilde{y}_{r_0}$ and $y_{\mathbf{r}}$, and applying Lemma \ref{lmm:MultipleSummation}, we find that (since $\xi\gg k(\log\log{x})^2/(\log{x})$)
\begin{align}
\sum_{\mathbf{r},r_0}&\frac{(y_{\mathbf{r},r_0})^2}{\prod_{p|r_0r}(p+O(k))}
\ll 
\Bigl(\sum_{\substack{r_0<x^\xi\\ (r_0,W_0)=1}}\frac{\tilde{y}_{r_0}^2}{\prod_{p|r_0}(p+O(k))}\Bigr)
\Bigl(\sum_{\mathbf{r}\in\mathcal{D}_k}\frac{y_{\mathbf{r}}^2}{\prod_{p|r}(p+O(k))}\Bigr)\nonumber\\
&\ll \xi(\log{R})^{k+1}\frac{W^{k}B^{k}W_0\mathfrak{S}_{WB}(\mathcal{L})^2}{\phi(WB)^{k}\phi(W_0)}\prod_{p\nmid W_0}\Bigl(1+\frac{O(k)}{p^2}\Bigr)\prod_{p\nmid W B}\Bigl(1+\frac{\omega(p)}{p+O(k)}\Bigr)\Bigl(1-\frac{1}{p}\Bigr)^k  I_k(F).\label{eq:S3MainExpression}
\end{align}
We note that the first product is $O(1)$ and the second product is $O(\mathfrak{S}_{WB}(\mathcal{L})^{-1})$,
since all primes in the products are greater than $2k^2$ and $\omega(p)\le k$. 
Thus, we obtain (recalling $\tilde{\lambda}_1\gg\xi\log{x}$)
\begin{equation}
\frac{\#\mathcal{A}(x)}{V\tilde{\lambda}_1^2}\sum_{\mathbf{r},r_0}\frac{(y_{\mathbf{r},r_0})^2}{\prod_{p|rr_0}(p+O(k))}\ll \xi^{-1} (\log{R})^{k-1}\#\mathcal{A}(x)\frac{W_0 W^k B^{k} \mathfrak{S}_{WB}(\mathcal{L})}{V\phi(W_0)\phi(W B)^{k}}I_k(F).
\end{equation}
We now sum over residue classes $v_0$ mod $V$, for which $L(v_0)$ is coprime to $\prod_{p\le 2k^2, p\nmid D}p$ and each of the $L_i(v_0)$ are coprime to $W$. The number $N$ of such residue classes is given by
\begin{equation}
N=\prod_{\substack{p|W\\ p\nmid D\Delta_L}}(p-\omega(p)-1)\prod_{\substack{p|W\\ p|D\Delta_L}}(p-\omega(p))\prod_{\substack{p|V/W\\ p\nmid Da_0}}(p-1)\prod_{\substack{p|V/W\\ p|Da_0}}p.
\end{equation}
This then gives,
\begin{align}
\sum_{n\in\mathcal{A}(x)}\mathbf{1}_{\mathcal{S}(\xi;D)}(L(n))w_n\ll \xi^{-1}\frac{B^k}{\phi(B)^k}(\log{R})^{k-1}\#\mathcal{A}(x)\mathfrak{S}_{WB}(\mathcal{L})I_k(F)\frac{N W_0 W^k}{V\phi(W_0)\phi(W)^k}.\label{eq:S4Bound2}
\end{align}
Finally, by calculation we find that
\begin{align}
\frac{N W_0 W^k}{V \phi(W_0) \phi(W)^k}&=\frac{\mathfrak{S}_B(\mathcal{L})\Delta_L D}{\mathfrak{S}_{WB}(\mathcal{L})\phi(\Delta_L D)}\prod_{\substack{p|(\Delta_L,V)\\p\nmid a_0WD}}\frac{p-1}{p}\prod_{\substack{p|W\\ p\nmid \Delta_LD}}\frac{(p-\omega(p)-1)p}{(p-\omega(p))(p-1)}\nonumber\\
&\le \frac{\mathfrak{S}_B(\mathcal{L})\Delta_L D}{\mathfrak{S}_{WB}(\mathcal{L})\phi(\Delta_L)\phi(D)}.
\end{align}
This gives the result.
\end{proof}
\begin{prpstn}\label{prpstn:S4}
Let $w_n$ be as described in Section \ref{sctn:InitialConsiderations}. For $L\in\mathcal{L}$ and $\rho\le \theta/10$, we have
\[\sum_{n\in\mathcal{A}(x)}\Bigl(\sum_{\substack{p|L(n)\\p<x^\rho\\ p\nmid B}}1\Bigr)w_n\ll \rho^2 k^4 (\log{k})^2 \mathfrak{S}_B(\mathcal{L})\#\mathcal{A}(x)(\log{R})^{k} I_{k}(F).\]
The implied constant depends only on $\theta,\alpha$ and the implied constants from \eqref{eq:Bdd1} and \eqref{eq:Bdd2}.
\end{prpstn}

\begin{proof}
We let $L(n)=L_m(n)=a_mn+b_m$ be the $m^{th}$ function in $\mathcal{L}$. As with Propositions \ref{prpstn:S1} and \ref{prpstn:S2}, we consider the sum restricted to $n\equiv v_0\pmod{W}$ for some $v_0$ with $(\prod_{i=1}^kL_i(v_0),W)=1$, since the other choices of $v_0$ make no contribution. This means we can also restrict the sum to $p\nmid W$.

Expanding the square and swapping the order of summation gives
\begin{equation}
\sum_{\substack{n\in\mathcal{A}(x)\\n\equiv v_0\pmod{W}}}\Bigl(\sum_{\substack{p|L(n)\\p<x^\rho\\ p\nmid W B}}1\Bigr)w_n=\sum_{\substack{p<x^\rho\\ p\nmid W B}}\sum_{\mathbf{d},\mathbf{e}\in\mathcal{D}_k}\lambda_\mathbf{d}\lambda_{\mathbf{e}}\sum_{\substack{n\in\mathcal{A}(x)\\ [d_i,e_i]|L_i(n)\\ n\equiv v_0\pmod{W}\\ p|L_m(n)}}1.
\end{equation}
The inner sum is empty unless $(d_ie_i,d_je_j)=1$ for all $i\ne j$ and $(d_ie_i,p)=1$ for all $i\ne m$. In this case, by the Chinese remainder theorem, we can combine the congruence conditions and see that the inner sum is $\#\mathcal{A}(x;q,a)$ for $q=[d_m,e_m,p]\prod_{i\ne m}[d_i,e_i]$ and some $a$. We write $\#\mathcal{A}(x;q,a)=\#\mathcal{A}(x)/q+O(E_q^{(1)})$ as in the proof of Propostion \ref{prpstn:S1}. We treat the error $E_q^{(1)}$ from making this change in the same manner as in Proposition \ref{prpstn:S1}, noting that all moduli $q$ we need to consider are square-free and satisfy $q<W R^2 x^\rho<x^\theta$, and for any $q$ there are $O(\tau_{3k+4}(q))$ choices of $\mathbf{d},\mathbf{e},p$ which give rise to the modulus $q$. Thus these error terms make a negligible contribution.

We use \eqref{eq:YDef} to change to our $y_{\mathbf{r}}$ variables, which gives us a main term of
\begin{equation}
\frac{\#\mathcal{A}(x)}{W}\sum_{\substack{p\le x^\rho\\ p\nmid W B}}\sideset{}{'}\sum_{\substack{\mathbf{d},\mathbf{e}\in\mathcal{D}_k\\ (d_ie_i,p)=1}}\frac{\lambda_\mathbf{d}\lambda_\mathbf{e}}{[d_m,e_m,p]\prod\limits_{i\ne m}[d_i,e_i]}=\frac{\#\mathcal{A}(x)}{W}\sum_{\substack{p<x^\rho \\ p\nmid W B}}\frac{1}{p}\sum_{\mathbf{r},\mathbf{s}\in\mathcal{D}_k}\frac{y_{\mathbf{r}}y_{\mathbf{s}}}{\phi_\omega(r)\phi_\omega(s)}\prod_{p'|rs}S_{p'}(\mathbf{r},\mathbf{s},p).\label{eq:S4MainTerm}
\end{equation}
Here if $p'\ne p$ then $S_{p'}(\mathbf{r},\mathbf{s},p)=S_{p'}(\mathbf{r},\mathbf{s})$, given by \eqref{eq:SpDef}, whereas if $p'=p$ we have
\begin{equation}
S_p(\mathbf{r},\mathbf{s},p)=\sideset{}{'}\sum_{\substack{\mathbf{d}|\mathbf{r},\mathbf{e}|\mathbf{s}\\ d_i,e_i|p\forall i\\ (d_ie_i,p)=1\forall i\ne m}}\frac{p\mu(d) \mu(e) d e}{[d_m,e_m,p] \prod_{i\ne m}[d_i,e_i] }
=
\begin{cases}
(p-1)^2,\qquad &p|(r_m,s_m),\\
-(p-1), &p|r_ms_m, p\nmid (r_m, s_m),\\
1, &p\nmid r_m s_m.
\end{cases}
\end{equation}
We let $\mathbf{u}=(r_1/(r_1,p),\dots,r_k/(r_k,p))$ be the vector formed by removing a possible factor of $p$ from the components of $\mathbf{r}$. We note that for a fixed choice $\mathbf{u},\mathbf{s}\in\mathcal{D}_k$ and $p\nmid W$ we have
\begin{equation}\label{eq:pCancellation}
\sum_{\substack{\mathbf{r}\in\mathcal{D}_k\\ r_i/(r_i,p)=u_i\forall i\\ (r_i,W_i)=1}}\frac{S_{p}(\mathbf{r},\mathbf{s},p)}{\phi_\omega(r)}=\frac{\mu((s_m,p))\phi((s_m,p))}{\phi_\omega(u)}\Bigl(1+\frac{\omega(p)-1}{p-\omega(p)}-\frac{p-1}{p-\omega(p)}\Bigr)=0.
\end{equation}
Here the first term in parentheses represents the contribution when $(r,p)=1$, the second term represents the contribution when $p|r$ but $p\nmid r_m$ (and so there are $\omega(p)-1$ choices of which index can be a multiple of $p$) and the final term represents the contribution when $p|r_m$.

We substitute $y_{\mathbf{r}}=y_{\mathbf{u}}+(y_{\mathbf{r}}-y_{\mathbf{u}})$ into our main term. By \eqref{eq:pCancellation} we find the $y_\mathbf{u}$ term makes a total contribution of 0, leaving only the contribution from $(y_{\mathbf{r}}-y_{\mathbf{u}})$. Similarly we let $\mathbf{v}$ be the vector obtained by removing a possible factor of $p$ from $\mathbf{s}$. We make the equivalent substitution $y_\mathbf{s}=y_\mathbf{v}+(y_\mathbf{s}-y_{\mathbf{v}})$, with the $y_{\mathbf{v}}$ term making no contribution. By Lemma \ref{lmm:YDifference} we have
\begin{equation}
(y_{\mathbf{r}}-y_{\mathbf{u}})(y_{\mathbf{s}}-y_{\mathbf{v}})\ll Y_{\mathbf{u}}Y_{\mathbf{v}}T_k^2(\log{p})^2/(\log{R})^2.
\end{equation}
Substituting this bound into our main term \eqref{eq:S4MainTerm}, we obtain the bound
\begin{equation}
\ll \frac{T_k^2 \#\mathcal{A}(x)}{W}\sum_{\substack{p<x^\rho \\ p\nmid W B}}\frac{1}{p}\Bigl(\frac{\log{p}}{\log{R}}\Bigr)^2\sum_{\substack{\mathbf{u},\mathbf{v}\\ (u,p)=(v,p)=1}}Y_{\mathbf{u}}Y_{\mathbf{v}}\prod_{p'|uv}|S_{p'}(\mathbf{u},\mathbf{v})|\sum_{\substack{\mathbf{r},\mathbf{s}\in\mathcal{D}_k\\ r_i/(r_i,p)=u_i\forall i\\ s_i/(s_i,p)=v_i\forall i}}\frac{|S_p(\mathbf{r},\mathbf{s},p)|}{\phi_\omega(r)\phi_\omega(s)}
\end{equation}
A calculation reveals that the inner sum is $O(\phi_\omega(u)^{-1}\phi_\omega(v)^{-1})$ for all $p\nmid W B$. This gives the bound
\begin{align}
&\ll \frac{T_k^2 \#\mathcal{A}(x)}{W}\sum_{\substack{p<x^\rho\\ p\nmid WB}}\frac{1}{p}\Bigl(\frac{\log{p}}{\log{R}}\Bigr)^2\sum_{\substack{\mathbf{u},\mathbf{v}\in\mathcal{D}_k\\ (u,p)=(v,p)=1}}\frac{Y_{\mathbf{u}}Y_{\mathbf{v}}}{\phi_\omega(u)\phi_\omega(v)}\prod_{p'|uv}|S_{p'}(\mathbf{u},\mathbf{v})|\nonumber\\
&\ll \frac{T_k^2 \rho^2\#\mathcal{A}(x)}{W}\sum_{\mathbf{u},\mathbf{v}\in\mathcal{D}_k}\frac{Y_{\mathbf{u}}^2+Y_{\mathbf{v}}^2}{\phi_\omega(u)\phi_\omega(v)}\prod_{p'|uv}|S_{p'}(\mathbf{u},\mathbf{v})|.
\end{align}
Here we have dropped the requirement that $(u,p)=(v,p)=1$ and used $Y_{\mathbf{u}}Y_{\mathbf{v}}\le Y_{\mathbf{u}}^2+Y_{\mathbf{v}}^2$ for an upper bound.

We recall from \eqref{eq:SpDef} that $S_{p'}(\mathbf{u},\mathbf{v})=0$ unless $u=v$. By multiplicativity and from the definition \eqref{eq:SpDef} of $S_{p'}(\mathbf{u},\mathbf{v})$, we find that given $\mathbf{u}\in\mathcal{D}_k$, we have
\begin{equation}
\sum_{\mathbf{v}\in\mathcal{D}_k}\frac{\prod_{p'|uv}|S_{p'}(\mathbf{u},\mathbf{v})|}{\phi_\omega(v)}=\prod_{p'|u}\Bigl(\sum_{\substack{\mathbf{w}\in\mathcal{D}_k \\ w_i|p'\,\forall i}}\frac{|S_{p'}(\mathbf{u},\mathbf{w})|}{\phi_\omega(w)}\Bigr)=\prod_{p'|u}\Bigl(\frac{p-1}{p-\omega(p)}+\frac{\omega(p)-1}{p-\omega(p)}\Bigr).
\end{equation}
(Here the first term in parentheses in the final product corresponds to the $\mathbf{w}$ such that $p|(\mathbf{u},\mathbf{w})$ and the second term to the $\omega(p)-1$ choices of $\mathbf{w}$ such that $p\nmid (\mathbf{u},\mathbf{w})$.) Thus, we find
\begin{align}
\sum_{\mathbf{u},\mathbf{v}\in\mathcal{D}_k}\frac{Y_{\mathbf{u}}^2+Y_{\mathbf{v}}^2}{\phi_\omega(u)\phi_\omega(v)}\prod_{p'|uv}|S_{p'}(\mathbf{u},\mathbf{v})|\ll \sum_{\mathbf{r}\in\mathcal{D}_k}\frac{Y_\mathbf{r}^2}{g(r)},
\end{align}
where $g$ is the multiplicative function defined by $g(p)=(p-\omega(p))^2/(p+\omega(p)-2)$. Applying Lemma \ref{lmm:MultipleSummation}, we see that this is
\begin{align}
&\ll \frac{B^k W^k \mathfrak{S}_{WB}(\mathcal{L})^2}{\phi(W B)^k}(\log{R})^k \prod_{p\nmid W B}\Bigl(1+\frac{\omega(p)}{g(p)}\Bigr)\Bigl(1-\frac{1}{p}\Bigr)^k I_k( F_2 ),
\label{eq:S4YrSum}
\end{align}
By Lemma \ref{lmm:IkJk} we have $I_k( F_2 )\ll k^2 I_k(F)$. Since any prime $p\nmid W B$ has $p>2k^2$ and $g(p)=p+O(k)$, we see the product is $\ll \mathfrak{S}_{W B}(\mathcal{L})^{-1}$. Thus \eqref{eq:S4YrSum} is
\begin{align}
&\ll k^2\frac{ B^k W^k \mathfrak{S}_{WB}(\mathcal{L})}{\phi(W B)^k}(\log{R})^k I_k(F).
\end{align}
Putting this all together gives
\begin{equation}
\sum_{\substack{n\in\mathcal{A}(x)\\n\equiv v_0\pmod{W}}}\Bigl(\sum_{\substack{p|L(n)\\p<x^\rho \\ p\nmid B}}1\Bigr)w_n\ll k^2T_k^2\rho^2\#\mathcal{A}(x) \frac{ B^k W^{k-1} \mathfrak{S}_{WB}(\mathcal{L})}{\phi(W B)^k}(\log{R})^k I_k(F).
\end{equation}
Summing over the $\phi_{\omega}(W)$ residue classes mod $W$ then gives the result.
\end{proof}
\section{Acknowledgements}
I would like to thank Sary Drappeau, Andrew Granville, Dimitris Koukoulopoulos and Jesse Thorner for many useful comments and suggestions. The author is funded by a CRM-ISM postdoctoral fellowship at the Universit\'e de Montr\'eal.
\bibliographystyle{plain}
\bibliography{Subsets}
\end{document}